\documentclass{amsart}

\usepackage{epsfig}
 \usepackage{amsmath}
 \usepackage{amssymb}
 \usepackage{amscd}
 \usepackage{graphicx}
 \usepackage{color}

\newtheorem{thm}{Theorem}[section]
 \newtheorem{lem}[thm]{Lemma}
 \newtheorem{prop}[thm]{Proposition}

 \newtheorem{cor}[thm]{Corollary}

 \newtheorem{rem}[thm]{Remark}

\def \N {\mathbb N}
 
 \def \Z {\mathbb Z}
 \def \R {\mathbb R}
 \def \E {\mathbb E}
 \def \Q {\mathbb Q}

\numberwithin{equation}{section}

\begin{document}

\title{sub-additive ergodic theorems for countable amenable groups}

\author{Anthony H. Dooley, Valentyn Ya. Golodets
  and Guohua Zhang}

\address{\vskip 2pt \hskip -12pt Anthony H. Dooley}

\address{\hskip -12pt Department of Mathematical Sciences, University of Bath, Bath,
 BA2 7AY, United Kingdom}

\email{a.h.dooley@bath.ac.uk}

\address{\vskip 2pt \hskip -12pt Valentin Ya. Golodets}

\address{\hskip -12pt School of Mathematics and Statistics, University of New South Wales, Sydney, NSW 2052, Australia}

\email{v.golodets@unsw.edu.au}

\address{\vskip 2pt \hskip -12pt Guohua Zhang}

\address{\hskip -12pt School of Mathematical Sciences and LMNS, Fudan University, Shanghai 200433, China}

\email{chiaths.zhang@gmail.com}

\begin{abstract}
 In this paper we  generalize Kingman's sub-additive ergodic theorem to a large class of infinite countable discrete amenable group actions.
 \end{abstract}

\maketitle

\markboth{A. H. Dooley, V. Ya. Golodets and G. H. Zhang}{sub-additive ergodic theorems for countable discrete amenable groups}


\section{Introduction}

The study of ergodic theorems was started in 1931 by von Neumann and Birkhoff, growing from problems in statistical mechanics. Ergodic theory soon earned its own place as an important part of functional analysis and probability, and grew into the study of measure-preserving transformations of a measure space.
 In 1968 an important new impetus to this area was received from Kingman's proof of the sub-additive ergodic theorem. This theorem opened up an impressive array of new applications \cite{Kingman68, Kingman73, Kingman76}. Krengel (\cite{Krengel}) showed that Kingman's theorem can be used to derive the multiplicative ergodic theorem of Oseledec \cite{Oseledec68}, which is of considerable current interest in the study of differentiable dynamical systems. Today, there are many elegant proofs of the theorem \cite{Derriennic75, KWei82, Kingman73, Neveu83, Smel77, Steele}. Among these, perhaps the shortest proof is that of Steele \cite{Steele}: this relies neither on a maximal inequality nor on a combinatorial Riesz lemma.
 A lovely exposition of the whole theory is given in Krengel's book \cite{Krengel}.

A statement of Kingman's sub-additive ergodic theorem is as follows:

\begin{thm} \label{1105231617}
 Let $\vartheta$ be a measure preserving transformation over the Lebesgue space $(Y, \mathcal{D}, \nu)$ and $\{f_n: n\in \N\}\subseteq L^1 (Y, \mathcal{D}, \nu)$ satisfy $f_{n+ m} (y)\le f_n (y)+ f_m (\vartheta^n y)$ for $\nu$-a.e. $y\in Y$ and all $n, m\in \N$. Then
 \begin{equation*}
 \lim_{n\rightarrow \infty} \frac{1}{n} f_n (y)= f (y)\ge - \infty
 \end{equation*}
  for $\nu$-a.e. $y\in Y$, where $f$ is an invariant measurable function over $(Y, \mathcal{D}, \nu)$.
 \end{thm}

If all the $f_n$ are constant functions, equal to $a_n$ (say), the theorem reduces to a well-known basic fact in analysis:
 if the sequence $\{a_n: n\in \N\}\subseteq \R$ satisfies $a_{n+ m}\le a_n+ a_m$ for all $n, m\in \N$, then
 \begin{equation*}
 \lim_{n\rightarrow \infty} \frac{a_n}{n}= \inf_{n\in \N} \frac{a_n}{n}\ge - \infty.
 \end{equation*}

It is an easy extension of Kingman's theorem that if $\inf_n \frac{\int f_n d\nu}{n} > - \infty$ , then the convergence also holds in $L^1$.

In this paper, we shall discuss extensions of Kingman's theorem to the class of countable discrete amenable groups.

The class of amenable groups
 includes all finite groups, solvable groups and compact groups, and actions of these groups on a Lebesgue space are a natural
 extension of the $\Z$-actions considered in Kingman's theorem: the foundations of the theory of amenable group actions were laid by Ornstein and Weiss in their pioneering paper \cite{OW}.

Lindenstrauss \cite{L} established the pointwise ergodic theorem for general locally
 compact amenable group actions along F\o lner sequences (with some natural conditions),
 which generalizes the Birkhoff theorem from $\Z$-actions to general amenable group actions, see also Benji Weiss' lovely survey article \cite{We}. For other related work, see
 \cite{AABBDGHJLMOSS10, BowenNevo, Cal53, E, E(1974), GE74, Nevo06, OW(1992), Shulman88, Tem72, Tem92}.

In contrast to the amenable group $\Z$,  a general infinite countable discrete amenable group
 may have a complicated combinatorial structure, and our challenge is to consider the limiting
 behaviour of the Kingman type in this context.

In general, we define a subset $\mathbf{D} = \{d_F: F \in
 \mathcal{F}_G\}$ of functions in $L^1(Y, \mathcal{D}, \nu)$, indexed
 by the family $\mathcal{F}_G$ of all non-empty finite subsets of $G$,
 to be \emph{$G$-invariant} and \emph{sub-additive} if $d_{E g} (y)=
 d_E (g y)$ and $d_{E\cup F} (y)\le d_E (y)+ d_F
  (y)$ for
  $\nu$-a.e. $y\in Y$, any $g\in G$ and all disjoint $E, F\in \mathcal{F}_G$.
  Then a natural generalization of Kingman's theorem is to ask whether, for an invariant
  sub-additive family, the limit
 \begin{equation}
 \lim_{n\rightarrow \infty} \frac{1}{|F_n|} d_{F_n} (y),
 \end{equation}
 exists for a F\o lner sequence $\{F_n: n\in \N\}$ of $G$, either
 pointwise almost everywhere, or, if $\liminf\limits_{n\rightarrow
 \infty} \frac{1}{|F_n|} \int d_{F_n} (y) d \nu (y) > - \infty$, in
 $L^1$.

This theorem reduces to Kingman's theorem for the F\o lner sequence
 $F_n = \{0, 1, \cdots , n- 1\}$, although it is not \emph{a priori}
 clear whether it holds for arbitrary good F\o lner sequences even in
 the integers. Our overall aim, for amenable groups, is to find
 conditions on the F\o lner sequence, and the group, under which this
 theorem holds.

 The first step, motivated by \cite{DZRDS} (in particular  \cite[Proposition 9.1 and Proposition 10.4]{DZRDS} and their proof), is to replace the limit by the limit superior.  Under some natural assumptions, we can prove versions of the Kingman theorem: precisely, if the family is either $G$-bi-invariant or strongly sub-additive, we show the existence of the $\limsup$. Moreover, if the group has the property of self-similarity (see Section \ref{techn}), then Kingman's theorem can be generalized completely to an action over a Lebesgue space. We shall see in the last section of the paper that this class of infinite countable discrete amenable groups includes many interesting groups.
   Observe that the assumption of self-similarity depends only on the algebraic structure of the group, whereas the conditions of bi-invariance and strong sub-additivity are properties of the particular family of functions chosen. As shown by \cite{DZRDS, MO}, strong sub-additivity plays an important role in the study of actions of an amenable group on Lebesgue space. In particular in his treatment of measure-theoretic entropy theory for the actions of an amenable group on a Lebesgue space,
  Moulin Ollagnier used the property of strong sub-additivity rather heavily (cf \cite[Chapter 4]{MO}).

  After we had submitted this paper in August, 2011 our attention was drawn to recent work of \cite{Pogor} by Pogorzelski, where a Banach space valued pointwise convergence for additive processes (\cite[Theorem 7.11]{Pogor}) is proved, together with an abstract mean ergodic theorem for set functions, based on a \emph{F\o lner vanishing invariant boundary term $b$}.
  
  In contrast to our assumption of subadditivity of $d$, the assumption of \cite{Pogor} is that the family of functions which he denotes by $F$, are $b$-\emph{almost additive} in the sense that $|F(Q) - \Sigma_{k=1}^m F(Q_k) | \leq \Sigma_{k=1}^m b(Q_k) $ for a disjoint union $Q= \cap_{j=1}^m Q_j$. (We here cite the special case where the Banach space is $\mathbb R$). The boundary term satisfies $\frac{b(U_j)}{|U_j|} \to 0$ as $j \to \infty$ for any F\o lner sequence $\{U_j\}$, is invariant under $G$, and compatible with unions and intersections. In a certain sense, $F$ becomes linear at infinity. 
  
  A mean ergodic theorem \cite[Theorem 5.7]{Pogor} is proved under compactness criteria for orbits induced by the actions on the Banach space.  Pogorzelski's conditions of additivity and $b$-almost additivity imply certain conditions on the group $G$ and on the invariant family $\mathcal F$.

 Our paper is organized as follows. Sections \ref{preli} and \ref{prepa} contain some preliminary remarks on infinite countable discrete amenable groups and pointwise ergodic theorems for their actions on a Lebesgue space. Motivated by the results of \cite{DZRDS}, in Section \ref{super} we analyze the  limit superior behavior of an invariant sub-additive family of integrable functions on an infinite countable discrete amenable group action, under some natural assumptions. In Section \ref{techn}, we present our main results for some special infinite countable discrete amenable groups. Finally in Section \ref{appli} we give some direct applications of the results obtained in previous sections.

We believe that the theorems we establish will lead to further developments along the lines of Oseledec' theorem and other applications.

\section{Preliminaries} \label{preli}
  Throughout, we assume that $G$ is a countably infinite discrete amenable group.
 We begin by recalling some basic properties of $G$.  These and many further details can be found in \cite{OW}.

Denote by $\mathcal{F}_G$ the
 set of all non-empty finite subsets of $G$.
 $G$ is called \emph{amenable}, if
 for each $K\in \mathcal{F}_G$ and any $\delta> 0$ there exists $F\in \mathcal{F}_G$
 such that
 $$|F\Delta K F|< \delta |F|,$$
 where $|\bullet|$ is counting measure on the set $\bullet$, $K F= \{k f: k\in K, f\in
 F\}$ and $F\Delta K F= (F\setminus K F)\cup (K F\setminus F)$.
 Let $K\in \mathcal{F}_G$ and $\delta> 0$. Set $K^{- 1}= \{k^{- 1}: k\in K\}$.
 $A\in \mathcal{F}_G$ is called \emph{$(K, \delta)$-invariant}, if
 $$|K^{- 1} A\cap K^{- 1} (G\setminus
 A)|< \delta |A|.$$
 A sequence $\{F_n: n\in \mathbb{N}\}$ in $\mathcal{F}_G$ is called
 a
 \emph{F\o lner sequence}, if for any $K\in \mathcal{F}_G$ and for any $\delta> 0$,
 $F_n$ is $(K, \delta)$-invariant whenever $n\in \mathbb{N}$ is sufficiently large,
 i.e.
 \begin{equation}
 \lim_{n\rightarrow \infty} \frac{|g F_n\Delta F_n|}{|F_n|}= 0
 \end{equation}
  for each $g\in G$.
 It is
 not hard to deduce the usual asymptotic invariance property:
 $G$ is amenable if and only if  $G$ has a F\o lner sequence
 $\{F_n: n\in \mathbb{N}\}$.

For example, for $G=\Z$, a
 F\o lner sequence is defined by $F_n=\{ 0,1,\cdots,n-1\}$, or, indeed, $\{ a_n,a_n+1,\cdots,a_n+n- 1 \}$ for any sequence $\{a_n: n\in
 \mathbb{N}\}\subseteq \Z$.

A \emph{measurable dynamical $G$-system} (MDS) $(Y, \mathcal{D}, \nu, G)$ is a Lebesgue space $(Y, \mathcal{D}, \nu)$  and a group $G$ of invertible measure preserving transformations of $(Y, \mathcal{D}, \nu)$ with $e_G$ acting as the identity transformation, where $e_G$ is the unit of the group $G$.

From now on $(Y, \mathcal{D}, \nu, G)$ will denote an MDS.

Let $\mathcal{I}$ be the sub-$\sigma$-algebra $\{D\in \mathcal{D}: \nu (g D\Delta D)= 0\ \text{for each}\ g\in G\}$, and for each $f\in L^1 (Y, \mathcal{D}, \nu)$ denote by $\E (f| \mathcal{I})$ the conditional expectation of $f$ over $\mathcal{I}$ with respect to $\nu$.
 If in addition, the MDS $(Y, \mathcal{D}, \nu, G)$ is \emph{ergodic}, (i.e. $\nu (D)$ is equal to either zero or one for all $D\in \mathcal{I}$), then
 $$\E (f| \mathcal{I}) (y)= \int_Y f (y') d \nu (y')$$
  for $\nu$-a.e. $y\in Y$.
 The measurable function $\E (f| \mathcal{I})$ is $G$-invariant and $\E (f| \mathcal{I})\in L^p (Y, \mathcal{D}, \nu)$ if $f\in L^p (Y, \mathcal{D}, \nu)$ for each $1\le p\le \infty$.

A sequence $\{E_n: n\in \mathbb{N}\}\subseteq \mathcal{F}_G$ is said
 to be \emph{tempered} if there exists $M> 0$ such that $|\bigcup\limits_{k= 1}^n E_k^{- 1} E_{n+ 1}|\le M |E_{n+ 1}|$ for each $n\in \mathbb{N}$. It is easy to show (see \cite{L}) that every F\o lner sequence of the group $G$  contains a tempered sub-sequence.

A related property which we will need in Lemma \ref{1103052244}  is the Tempelman condition:
 a sequence $\{E_n: n\in \mathbb{N}\}\subseteq \mathcal{F}_G$ is said
 to satisfy the \emph{Tempelman condition} if there exists $M> 0$ such that $|\bigcup\limits_{k= 1}^n E_k^{- 1} E_n|\le M |E_n|$ for each $n\in \mathbb{N}$. It is easy to see that every sequence satisfying the Tempelman condition is a tempered sequence, but the converse is not true: indeed the lamplighter group (see \cite{L}) is a countable discrete amenable group where no F\o lner sequence contains a sub-sequence satisfying the Tempelman condition.

Lindenstrauss (\cite[Theorem 1.2]{L} and \cite{Lannounce, We}) showed how to use tempered sequences to generalize the Birkhoff pointwise convergence theorem to an amenable group action as follows:

\begin{thm} \label{1006131615}
 Let $f\in L^p (Y, \mathcal{D}, \nu)$ and $\{F_n: n\in \mathbb{N}\}$ be a tempered F\o lner sequence of $G$, where $1\le p< \infty$. Then
 \begin{equation*}
 \lim_{n\rightarrow \infty} \frac{1}{|F_n|} \sum_{g\in F_n} f (g y)= \E (f| \mathcal{I}) (y)
 \end{equation*}
 for $\nu$-a.e. $y\in Y$ and in the sense of $L^p$.
 \end{thm}

\begin{rem} \label{1103071447}

Lindenstrauss actually states the conclusion for pointwise convergence in the case of $p= 1$. As remarked in \cite{Lannounce}, mean convergence easily follows by methods familiar from $\Z$-actions.

 Moreover, as shown in \cite{JR79}, a restriction (such as temperance) on the F\o lner sequence is essential even in the special case $G= \Z$.

Note further that \cite[Theorem 2.1]{We} discusses mean convergence in the case of $p= 2$, and shows that this holds for any F\o lner sequence.
 \end{rem}

A function $f\in L^1 (Y, \mathcal{D}, \nu)$ is \emph{non-negative} if $f (y)\ge 0$ for
  $\nu$-a.e. $y\in Y$; and \emph{$G$-invariant} if $ f (g y)= f (y)$   for $\nu$-a.e. $y\in Y$ and all $g\in G$.

 Let $\mathbf{D}= \{d_F: F\in \mathcal{F}_G\}$ be a \emph{family} of functions in $L^1 (Y, \mathcal{D}, \nu)$. We make some natural assumptions on the family $\mathbf{D}$ which will enable us to state analogues of Theorem \ref{1105231617}.

We say that $\mathbf{D}$ is:
  \begin{enumerate}

   \item
  \emph{non-negative} if: each element from $\mathbf{D}$ is non-negative;

   \item \emph{$G$-invariant} if: $d_{F g} (y)= d_F (g y)$ for $\nu$-a.e. $y\in Y$ and all $F\in \mathcal{F}_G, g\in G$;

   \item \emph{$G$-bi-invariant} if: it is $G$-invariant, and in addition $d_{gF}(y) = d_{Fg}(y)$ for $\nu$-a.e. $y\in Y$ and all $F\in \mathcal{F}_G, g\in G$;

   \item \emph{monotone} if: $d_E (y)\le d_F (y)$ for $\nu$-a.e. $y\in Y$ and all $\emptyset\neq E\subseteq F\in \mathcal{F}_G$;

   \item
 \emph{sub-additive} (\emph{sup-additive}, respectively) if: $d_{E\cup F} (y)\le d_E (y)+ d_F
  (y)$ for
  $\nu$-a.e. $y\in Y$ and all disjoint $E, F\in \mathcal{F}_G$ ($- \mathbf{D}$ is sub-additive, respectively);

 \item \emph{strongly sub-additive} (\emph{strongly sup-additive}, respectively) if $d_{E\cap F} (y)+ d_{E\cup F} (y)$ $\le d_E (y)+ d_F (y)$ for $\nu$-a.e. $y\in Y$ and all $E, F\in \mathcal{F}_G$, here by convention we set $d_\emptyset (y)= 0$ for each $y\in Y$ ($- \mathbf{D}$ is strongly sub-additive, respectively).
  \end{enumerate}

 For example, for each $f\in
  L^1 (Y, \mathcal{D}, \nu)$, it is easy to check that
  $$\mathbf{D}^f\doteq \{d_F^f (y)\doteq \sum_{g\in F} f (g y): F\in \mathcal{F}_G\}$$
   is a strongly sub-additive $G$-invariant family in $L^1 (Y, \mathcal{D}, \nu)$.

   Of course, if $G$ is abelian, then every $G$-invariant family is $G$-bi-invariant; this is also clearly true if for every $F\in \mathcal{F}_G$, the function $d_{F}$ depends only on the cardinality of $F.$

In fact, let $\mathbf{D}= \{d_F: F\in \mathcal{F}_G\}\subseteq L^1 (Y, \mathcal{D}, \nu)$ be a strongly sub-additive family. If $\mathcal{D}$ is $G$-invariant (respectively $G$-bi-invariant) then the family
$\{d_F (y)- |F|^2: F\in \mathcal{F}_G\}\subseteq L^1 (Y, \mathcal{D}, \nu)$
is also $G$-invariant (respectively $G$-bi-invariant).

Let $\mathbf{D}= \{d_F: F\in \mathcal{F}_G\}\subseteq L^1 (Y, \mathcal{D}, \nu)$ be a sub-additive $G$-invariant family. We are interested in the
 convergence of
 \begin{equation} \label{1105231650}
 \lim_{n\rightarrow \infty} \frac{1}{|F_n|} d_{F_n} (y),
 \end{equation}
 where $\{F_n: n\in \N\}$ is a F\o lner sequence of $G$.

These are the basic ingredients of our version of Kingman's sub-additive ergodic theorem for infinite countable discrete amenable group actions.

\section{Preparations} \label{prepa}

In this section, we consider the asymptotic limiting behaviour of \eqref{1105231650} in the simplest case, when all functions $f_n$ are constant functions. We also present a version of the maximal inequality which will be used in later sections following the ideas of \cite{Krengel}.

In contrast to a $\Z$-action, for a general infinite countable discrete amenable group action it is not clear whether the limit \eqref{1105231650} exists, even if all the functions in $\mathbf{D}$ are constant functions. Thus, we will add some natural conditions.

Let $\emptyset\neq T\subseteq G$. We say that \emph{$T$ tiles $G$} if there exists $\emptyset\neq G_T\subseteq G$ such that $\{T c: c\in G_T\}$ forms a partition of $G$, that is, $T c_1\cap T c_2= \emptyset$ if $c_1$ and $c_2$ are different elements from $G_T$ and
  $\bigcup\limits_{c\in G_T} T c= G$.

Denote by $\mathcal{T}_G$ the set of all non-empty finite subsets of $G$ which tile $G$. Observe that $\mathcal{T}_G\neq \emptyset$, as $\mathcal{T}_G\supseteq \{\{g\}: g\in G\}$.

Let $(Y, \mathcal{D}, \nu, G)$ be an MDS. We say that \emph{$G$ acts freely on $(Y, \mathcal{D}, \nu)$} if $\{y\in Y: g y= y\}$ has zero $\nu$-measure for any $g\in G\setminus \{e_G\}$.

Tiling sets play a key role in establishing a version of Rokhlin's Lemma for infinite countable discrete amenable group actions (see \cite[Theorem 3.3 and Proposition 3.6]{We}). In particular, we have:

\begin{prop} \label{1008272250}
 Let $T\in \mathcal{F}_G$. Then
 $T\in \mathcal{T}_G$ if and only if, for every MDS $(Y, \mathcal{D}, \nu, G)$, where $G$ acts freely on $(Y, \mathcal{D}, \nu)$, for each $\epsilon> 0$ there exists $B\in \mathcal{D}$ such that the family $\{t B: t\in T\}$ is disjoint and $\nu (\bigcup\limits_{t\in T} t B)\ge 1- \epsilon$.
 \end{prop}

The class of countable amenable groups admitting a \emph{tiling F\o
 lner sequence} (i.e. a F\o lner sequence consisting of tiling
 subsets of the group) is large, and includes all countable amenable
 linear groups and all countable residually finite amenable groups
 \cite{W0}. Recall that a \emph{linear group} is an abstract group
 which is isomorphic to a matrix group over a field $K$ (i.e. a group
 consisting of invertible matrices over some field $K$); a group is
 \emph{residually finite} if the intersection of all its normal
 subgroups of finite index is trivial. Note that any finitely
 generated nilpotent group is residually finite. The question of
 whether every countable discrete amenable group admits a tiling F\o
 lner sequence remains open \cite{OW}.

The following results are  \cite[Proposition 2.5 and Proposition 2.3]{DZRDS}.  See also \cite[Theorem 6.1]{Lin-Wei}, \cite[Theorem 5.9]{We} and \cite{MO, WZ}.

 \begin{prop} \label{1103051237}
 Let $f: \mathcal{F}_G\rightarrow \mathbb{R}$ be a function. Assume that $f (E g)= f (E)$ and $f (E\cup F)\le f (E)+ f (F)$ whenever $g\in G$ and $E, F\in \mathcal{F}_G$ satisfy $E\cap F= \emptyset$. Then for any tiling F\o lner sequence $\{F_n: n\in \mathbb{N}\}$ of $G$, the sequence $\{\frac{f (F_n)}{|F_n|}: n\in \mathbb{N}\}$ converges and the value of the limit is independent of the selection of the tiling F\o lner sequence $\{F_n: n\in \mathbb{N}\}$, in fact:
 \begin{equation*}
 \lim_{n\rightarrow \infty} \frac{f (F_n)}{|F_n|}= \inf_{F\in \mathcal{T}_G} \frac{f (F)}{|F|}\ (\text{and so}\ = \inf_{n\in \mathbb{N}} \frac{f (F_n)}{|F_n|}).
 \end{equation*}
 \end{prop}

\begin{prop} \label{1102111944}
 Let $f: \mathcal{F}_G\rightarrow \mathbb{R}$ be a function. Assume that $f (E g)= f (E)$ and $f (E\cap F)+ f (E\cup F)\le f (E)+ f (F)$ whenever $g\in G$ and $E, F\in \mathcal{F}_G$ (here, we set $f (\emptyset)= 0$ by convention). Then for any F\o lner sequence $\{F_n: n\in \mathbb{N}\}$ of $G$, the sequence $\{\frac{f (F_n)}{|F_n|}: n\in \mathbb{N}\}$ converges and the value of the limit is independent of the selection of the F\o lner sequence $\{F_n: n\in \mathbb{N}\}$, in fact:
 \begin{equation*}
 \lim_{n\rightarrow \infty} \frac{f (F_n)}{|F_n|}= \inf_{F\in \mathcal{F}_G} \frac{f (F)}{|F|}\ (\text{and so}\ = \inf_{n\in \mathbb{N}} \frac{f (F_n)}{|F_n|}).
 \end{equation*}
 \end{prop}

The difference between Proposition \ref{1103051237} and Proposition \ref{1102111944} was shown in  \cite[Example 2.7]{DZRDS} in the special case of $G= \Z$.

   Let $\mathbf{D}= \{d_F: F\in \mathcal{F}_G\}\subseteq L^1 (Y, \mathcal{D}, \nu)$ be a sub-additive $G$-invariant family. By Proposition \ref{1103051237} the limit
   \begin{equation} \label{1105231544}
   \lim_{n\rightarrow \infty} \frac{1}{|F_n|} \int_Y d_{F_n} (y) d \nu (y)
   \end{equation}
   exists
   for (and independent of) any tiling F\o lner sequence $\{F_n: n\in \N\}$ of $G$, and is equal to
   \begin{equation*}
   \inf_{n\in \N} \frac{1}{|F_n|} \int_Y d_{F_n} (y) d \nu (y)= \inf_{T\in \mathcal{T}_G} \frac{1}{|T|} \int_Y d_T (y) d \nu (y)< \infty.
   \end{equation*}
 Now if $\mathbf{D}$ is a strongly sub-additive $G$-invariant family, alternatively using Proposition \ref{1102111944}
 the limit \eqref{1105231544} exists for (and independent of) any F\o lner sequence $\{F_n: n\in \N\}$ of $G$, and is equal to
   \begin{equation*}
 \inf_{n\in \N} \frac{1}{|F_n|} \int_Y d_{F_n} (y) d \nu (y)= \inf_{F\in \mathcal{F}_G} \frac{1}{|F|} \int_Y d_F (y) d \nu (y)< \infty.
   \end{equation*}
 Dually, if $\mathbf{D}$ is a sup-additive or strongly sup-additive $G$-invariant family, we can talk about the limit similarly.
    We will denote by $\nu (\mathbf{D})$ these limits in the sequel. Remark that they need not to be a finite constant.

   Recall that the sub-$\sigma$-algebra $\mathcal{I}$ is introduced in previous section as $\{D\in \mathcal{D}: \nu (g D\Delta D)= 0\ \text{for each}\ g\in G\}$. Let $\nu= \int_Y \nu_y d \nu (y)$ be the disintegration of $\nu$ over $\mathcal{I}$. Then $(Y, \mathcal{D}, \nu_y, G)$ will be an ergodic MDS for $\nu$-a.e. $y\in Y$. The disintegration is the \emph{ergodic decomposition of $\nu$} (cf
    \cite[Theorem 3.22]{G1}), which can be characterized as follows: for each $f\in L^1 (Y, \mathcal{D}, \nu)$, $f\in L^1 (Y, \mathcal{D}, \nu_y)$ for $\nu$-a.e. $y\in Y$ and the function $y\mapsto \int_Y f d \nu_y$ is in $L^1 (Y, \mathcal{I}, \nu)$. In fact,
     $\int_Y f d \nu_y= \E (f| \mathcal{I}) (y)$ for $\nu$-a.e. $y\in Y$ and hence $\int_Y (\int_Y f d \nu_y) d \nu (y)= \int_Y f d \nu$.

   Thus, we have:

   \begin{prop} \label{ergodic decomposition}
 Let $\mathbf{D}= \{d_F: F\in \mathcal{F}_G\}\subseteq L^1 (Y, \mathcal{D}, \nu)$ be a sub-additive (or strongly sub-additive, sup-additive, strongly sup-additive, and so on) $G$-invariant family. Assume that $\nu= \int_Y \nu_y d \nu (y)$ is the ergodic decomposition of $\nu$. Then
 \begin{equation} \label{1106012236}
 \nu (\mathbf{D})= \int_Y \nu_y (\mathbf{D}) d \nu (y).
 \end{equation}
    \end{prop}
    \begin{proof}
    Dually, we only consider the case of $\mathbf{D}$ being sup-additive and $\{F_n: n\in \N\}$ a tiling F\o lner sequence of $G$. And so $\nu (\mathbf{D})> -\infty$. Observe that $\int_Y d_F d \nu= \int_Y (\int_Y d_F d \nu_y) d \nu (y)$ for each $F\in \mathcal{F}_G$, then: on one hand,
    $$\nu (\mathbf{D})= \sup_{T\in \mathcal{T}_G} \frac{1}{|T|} \int_Y d_T d \nu\le \int_Y (\sup_{T\in \mathcal{T}_G} \frac{1}{|T|} \int_Y d_T d \nu_y) d \nu (y)= \int_Y \nu_y (\mathbf{D}) d \nu (y);$$
    on the other hand, for
    $$d_{F_n}' (y)= d_{F_n} (y)- \sum_{g\in F_n} d_{\{e_G\}} (g y)\ge 0,$$
    using Fatou's Lemma one has
    \begin{eqnarray*}
   & & \int_Y \nu_y (\mathbf{D}) d \nu (y)- \int_Y (\int_Y d_{\{e_G\}} d \nu_y) d \nu (y) \\
   &= & \int_Y (\liminf_{n\rightarrow \infty} \frac{1}{|F_n|} \int_Y d_{F_n}' d \nu_y) d \nu (y) \\
   &\le & \liminf_{n\rightarrow \infty} \frac{1}{|F_n|} \int_Y (\int_Y d_{F_n}' d \nu_y) d \nu (y)= \nu (\mathbf{D})- \int_Y d_{\{e_G\}} d \nu,
    \end{eqnarray*}
    equivalently, $\int_Y \nu_y (\mathbf{D}) d \nu (y)\le \nu (\mathbf{D})$. Summing up, we obtain \eqref{1106012236}.
    \end{proof}

In the remainder of the section, we present a version of the maximal
 inequality following the ideas of \cite{Krengel}, which will be used
 in later sections.

The following result is a version of \cite[Chapter 6, Theorem 4.2]{Krengel}: in fact, Krengel  states a version compatible with \eqref{1103072017}, while we need an equivalent version in the style of equation \eqref{1103072016}. We present a proof of it here for completeness:  the ideas are taken from \cite{Krengel}.

\begin{lem} \label{1103052244}
Let $G$ be a countable discrete amenable group.
 Assume that $\mathbf{D}= \{d_F: F\in \mathcal{F}_G\}\subseteq L^1 (Y, \mathcal{D}, \nu)$ is a non-negative sup-additive $G$-invariant family and $\{F_n: n\in \mathbb{N}\}$ is a F\o lner sequence $G$ satisfying the Tempelman condition with constant $M> 0$. Then
 \begin{eqnarray}
 \nu (Y_{\alpha, N})&\le & \frac{M}{\alpha}\cdot \inf_{n\in \N} \frac{1}{|F_n\cap \bigcap\limits_{i= 1}^N \bigcap\limits_{g\in F_i} g^{- 1} F_n|} \int_Y d_{F_n} (y) d \nu (y) \label{1103072018} \\
 &\le & \frac{M}{\alpha}\cdot \liminf_{n\rightarrow \infty} \frac{1}{|F_n|} \int_Y d_{F_n} (y) d \nu (y) \label{1103072016} \\
 &\le & \frac{M}{\alpha}\cdot \sup_{F\in \mathcal{F}_G} \frac{1}{|F|} \int_Y d_F (y) d \nu (y) \label{1103072017}
 \end{eqnarray}
 for each $\alpha> 0$ and any $N\in \N$,
 where
 $$Y_{\alpha, N}= \{y\in Y: \max_{k= 1, \cdots, N} \frac{d_{F_k} (y)}{|F_k|} > \alpha\}.$$
 \end{lem}
 \begin{proof}
 In fact, the proof will be complete once we prove \eqref{1103072018}.

We may set $d_\emptyset (y)= 0$ for each $y\in Y$ without
 any loss of generality.

Let $y\in Y$ and $n\in \N$ such that $F_n^*\doteq F_n\cap \bigcap\limits_{i= 1}^N \bigcap\limits_{g\in F_i} g^{- 1} F_n\neq \emptyset$. As $\{F_m: m\in \mathbb{N}\}$ is a F\o lner sequence of $G$, $F_m^*\neq \emptyset$ once $m\in \N$ is large enough. Set
 $$C_1= \{g\in F_n^*: \frac{d_{F_1} (g y)}{|F_1|}> \alpha\},$$
 and for $j= 2, \cdots, N$, set
 $$C_j= \{g\in F_n^*: \frac{d_{F_j} (g y)}{|F_j|}> \alpha\}\setminus \bigcup_{i= 1}^{j- 1} C_i.$$

Now let $C_N'\subseteq C_N$ be a maximal disjoint family in $\{F_N c: c\in C_N\}$. In particular,
  $$C_N\subseteq F_N^{- 1} F_N C_N'.$$
  Once the $C_N'\subseteq C_N, \cdots, C_i'\subseteq C_i$ have been constructed for some $i> 1$, then let $C_{i- 1}'\subseteq C_{i- 1}$ be a maximal disjoint family in $\{F_{i- 1} c: c\in C_{i- 1}\}$ which is also disjoint from the elements in $\{F_j c_j: c_j\in C_j', j= i, \cdots, N\}$. In particular,
  $$C_{i- 1}\subseteq F_{i- 1}^{- 1} (\bigcup_{j= i- 1}^N F_j C_j').$$

 From the construction we have
  $$\{g\in F_n^*: g y\in Y_{\alpha, N}\}= \bigcup_{i= 1}^N C_i\subseteq \bigcup_{i= 1}^N \bigcup_{j= 1}^i F_j^{- 1} F_i C_i'.$$
  Moreover, as the sequence $\{F_n: n\in \mathbb{N}\}$ satisfies the Tempelman condition, one has
  \begin{equation} \label{1103031503}
  |\{g\in F_n^*: g y\in Y_{\alpha, N}\}|\le \sum_{i= 1}^N |\bigcup_{j= 1}^i F_j^{- 1} F_i|\cdot |C_i'|\le M \sum_{i= 1}^N |F_i|\cdot |C_i'|.
  \end{equation}

As $\mathbf{D}= \{d_F: F\in \mathcal{F}_G\}\subseteq L^1 (Y, \mathcal{D}, \nu)$ is a non-negative sup-additive (and so monotone) $G$-invariant family, by the construction of $C_i, C_i', i= 1, \cdots, N$, one has
 \begin{eqnarray} \label{1103031606}
 d_{F_n} (y)\ge d_{\bigcup\limits_{i= 1}^N F_i C_i'} (y)&\ge & \sum_{i= 1}^N \sum_{g_i\in C_i'} d_{F_i g_i} (y)\nonumber \\
 &> & \alpha \sum_{i= 1}^N |F_i|\cdot |C_i'|\nonumber \\
 &\ge & \frac{\alpha}{M} |\{g\in F_n^*: g y\in Y_{\alpha, N}\}|\ (\text{using \eqref{1103031503}})
 \end{eqnarray}
 for $\nu$-a.e. $y\in Y$.
 Thus applying Fubini's Theorem we obtain
 \begin{eqnarray*} \label{1103031613}
 \frac{1}{|F_n^*|} \int_Y d_{F_n} (y) d \nu (y)&\ge & \frac{1}{|F_n^*|}\cdot \frac{\alpha}{M} \int_Y |\{g\in F_n^*: g y\in Y_{\alpha, N}\}| d \nu (y)\ (\text{using \eqref{1103031606}})\nonumber \\
 &= & \frac{1}{|F_n^*|}\cdot \frac{\alpha}{M} \sum_{g\in F_n^*} \int_Y 1_{Y_{\alpha, N}} (g y) d \nu (y)= \frac{\alpha}{M}\cdot \nu (Y_{\alpha, N}),
 \end{eqnarray*}
 which implies \eqref{1103072018}.
 This finishes our proof.
 \end{proof}

\begin{rem} \label{1103072309}
  On the basis of \cite[Theorem 3.2]{L} or \cite[Section 7]{We}, we conjecture that it is possible to  strengthen Lemma \ref{1103052244} by replacing the assumption that the  F\o lner sequence satisfies the Tempelman condition by the assumption that it is a tempered F\o lner sequence. We have not so far been able to prove this, however.
 \end{rem}

As a direct corollary of \eqref{1103072016} and \eqref{1103072017}, we have:

\begin{cor} \label{1103072031}
 Assume that $\mathbf{D}= \{d_F: F\in \mathcal{F}_G\}\subseteq L^1 (Y, \mathcal{D}, \nu)$ is a non-negative sup-additive $G$-invariant family and $\{F_n: n\in \mathbb{N}\}$ is a F\o lner sequence of $G$ satisfying Tempelman condition with constant $M> 0$. If either each $F_n, n\in \N$ tiles $G$, or the family $\mathbf{D}$ is strongly sup-additive, then, for any $\alpha> 0$,
  $$\nu (\{y\in Y: \sup_{k\in \N} \frac{d_{F_k} (y)}{|F_k|} > \alpha\})\le \frac{M}{\alpha} \nu (\mathbf{D}).$$
 \end{cor}

\section{The superior limit behavior of the family} \label{super}

Motivated by \cite[Proposition 9.1 and Proposition 10.4]{DZRDS} and their proofs, we aim to study the behaviour of the $\limsup$ in \eqref{1105231650}.  Our main results are Theorem \ref{1103061916} and Theorem \ref{1102001916}. The conclusion of Theorem \ref{1103061916} requires bi-invariance of the family $\mathbf{D}$, and the conclusion of Theorem \ref{1102001916} requires that it is strongly sub-additive.

In general it is not easy to understand the limit behavior of a family of integrable functions. As shown by \cite[Chapter 9 and Chapter 10]{DZRDS}, both bi-invariance and strong sub-additivity provide cases where we can study assumption $(\spadesuit)$ of \cite[Theorem 7.1]{DZRDS}, which reflects a kind of limiting behaviour. Property $(\spadesuit)$ appears naturally when one considers actions of the integers on a Lebesgue space, and strong sub-additivity has proven its importance in the study of measure-theoretic entropy theory for the actions of an amenable groups on a Lebesgue spaces (cf \cite{DZRDS, MO}).

The main results of this section are as follows.

\begin{thm} \label{1103061916}
 Let $\mathbf{D}= \{d_F: F\in \mathcal{F}_G\}\subseteq L^1 (Y, \mathcal{D}, \nu)$ be a sub-additive $G$-bi-invariant family and $\{F_n: n\in \mathbb{N}\}$ a tempered tiling F\o lner sequence of $G$. Then
  \begin{equation} \label{1103081924}
  \limsup_{n\rightarrow \infty} \frac{1}{|F_n|} d_{F_n} (y)= \inf_{T\in \mathcal{T}_G} \frac{1}{|T|} \E (d_T| \mathcal{I}) (y)
  \end{equation}
   for $\nu$-a.e. $y\in Y$,
  which is an invariant measurable function over $(Y, \mathcal{D}, \nu)$, and
  \begin{equation} \label{1103081657}
  \int_Y \inf_{T\in \mathcal{T}_G} \frac{1}{|T|} \E (d_T| \mathcal{I}) (y) d \nu (y)= \nu (\mathbf{D}).
  \end{equation}
   \end{thm}

 \begin{thm} \label{1102001916}
 Let $\mathbf{D}= \{d_F: F\in \mathcal{F}_G\}\subseteq L^1 (Y, \mathcal{D}, \nu)$ be a strongly sub-additive $G$-invariant family and $\{F_n: n\in \mathbb{N}\}$ a tempered F\o lner sequence of $G$.
  Then
  \begin{equation} \label{1102001657}
  \limsup_{n\rightarrow \infty} \frac{1}{|F_n|} d_{F_n} (y)= \inf_{F\in \mathcal{F}_G} \frac{1}{|F|} \E (d_F| \mathcal{I}) (y)
  \end{equation}
 for $\nu$-a.e. $y\in Y$, which is an invariant measurable function over $(Y, \mathcal{D}, \nu)$, and
  \begin{equation} \label{11103081924}
  \int_Y \inf_{F\in \mathcal{F}_G} \frac{1}{|F|} \E (d_F| \mathcal{I}) (y) d \nu (y)= \nu (\mathbf{D}).
  \end{equation}
  \end{thm}

 Before proceeding, we need the following easy observation.

 \begin{lem} \label{1103082044}
  Let $\{F_n: n\in \N\}$ be a tempered F\o lner sequence of $G$. Assume that $\emptyset\neq E_n\subseteq F_n$ for each $n\in \N$ satisfying $\lim\limits_{n\rightarrow \infty} \frac{|E_n|}{|F_n|}= 1$. Then $\{E_n: n\in \N\}$ is also a tempered F\o lner sequence of $G$.
  \end{lem}
  \begin{proof}
 Combined with the assumptions, the conclusion follows from the facts that
 \begin{equation*}
 E_n\Delta g E_n\subseteq F_n\Delta g F_n\cup \{g, e_G\} (F_n\setminus E_n)
 \end{equation*}
 and
 \begin{equation*}
 \bigcup_{i= 1}^n E_i^{- 1} E_{n+ 1}\subseteq \bigcup_{i= 1}^n F_i^{- 1} F_{n+ 1}
 \end{equation*}
 for each $n\in \N$ and $g\in G$, which are easy to check.
  \end{proof}

The following result from \cite[Lemma 10.5]{DZRDS} is used in the
 proof of Theorem \ref{1103061916}. Observe that whilst the whole
 discussion of  \cite[Chapter 10]{DZRDS} is in the setting where $\{F_n:
 n\in \N\}$ is an increasing tiling F\o lner sequence of $G$, the
 proof of \cite[Lemma 10.5]{DZRDS} uses only the fact that $\{F_n:
 n\in \N\}$ is a F\o lner sequence of $G$. The lemma is stated under the hypothesis that $G$ is abelian,
 but this property is used only in the formula $d_{Fg}(y) = d_{gF}(y) = d_F(gy)$ to prove inequality (10.5) of \cite{DZRDS}.
 Thus the result holds for an arbitrary countable discrete amenable group, provided that we assume $G$-bi-invariance of the family $\mathbf{D}= \{d_F: F\in \mathcal{F}_G\}\subseteq L^1 (Y, \mathcal{D}, \nu)$.

We state this result as:

\begin{lem} \label{01103081701}
 Let $\mathbf{D}= \{d_F: F\in \mathcal{F}_G\}\subseteq L^1 (Y, \mathcal{D}, \nu)$ be a sub-additive $G$-invariant family, $\{F_n: n\in \N\}$ a F\o lner sequence of $G$ and $T\in \mathcal{T}_G, \epsilon> 0$. Assume the family $- \mathbf{D}$ is non-negative and $G$-bi-invariant. Then, whenever $n\in \N$ is large enough, there exists $H_n\subseteq F_n$ such that $|F_n\setminus H_n|< 2 \epsilon |F_n|$ and, for $\nu$-a.e. $y\in Y$,
 $$d_{F_n} (y)\le \frac{1}{|T|} \sum_{g\in H_n} d_T (g y).$$
 \end{lem}

Now we can prove Theorem \ref{1103061916}.

\begin{proof}[Proof of Theorem \ref{1103061916}]
 As $\mathbf{D}$ is a sub-additive $G$-invariant family, the function
 $$\sum_{g\in F_n} d_{\{e_G\}} (g y)- d_{F_n} (y)$$
 is non-negative for each $n\in \N$, and observe that the sequence
 $\{F_n: n\in \N\}$ is a tempered tiling F\o lner sequence of $G$, by Fatou's Lemma, one has
 \begin{eqnarray*}
 & & \int_Y d_{\{e_G\}} (y) d \nu (y)- \nu (\mathbf{D})= \liminf_{n\rightarrow \infty} \int_Y \frac{1}{|F_n|} [\sum_{g\in F_n} d_{\{e_G\}} (g y)- d_{F_n} (y)] d \nu (y)\nonumber \\
 &\ge & \int_Y \liminf_{n\rightarrow \infty} \frac{1}{|F_n|} [\sum_{g\in F_n} d_{\{e_G\}} (g y)- d_{F_n} (y)] d \nu (y)\nonumber \\
 &= & \int_Y [\E (d_{\{e_G\}}| \mathcal{I}) (y)- \limsup_{n\rightarrow \infty} \frac{1}{|F_n|} d_{F_n} (y)] d \nu (y)\ \text{(using Theorem \ref{1006131615})}\nonumber \\
 &= & \int_Y d_{\{e_G\}} (y) d \nu (y)- \int_Y \limsup_{n\rightarrow \infty} \frac{1}{|F_n|} d_{F_n} (y) d \nu (y),
 \end{eqnarray*}
 which implies
 \begin{equation} \label{1103081842}
 \int_Y \limsup_{n\rightarrow \infty} \frac{1}{|F_n|} d_{F_n} (y) d \nu (y)\ge \nu (\mathbf{D}).
 \end{equation}

By assumption,  the family $- \mathbf{D}$ is non-negative.
 Now applying Lemma \ref{01103081701} to $\mathbf{D}$ we obtain that, once $T\in \mathcal{T}_G$ (fixed) and $\epsilon> 0$, if $n\in \mathbb{N}$ is large enough then there exists $T_n\subseteq F_n$ such that $|F_n\setminus T_n|\le 2 \epsilon |F_n|$ and, for $\nu$-a.e. $y\in Y$,
 \begin{equation} \label{1008292331}
 d_{F_n} (y)\le \frac{1}{|T|} \sum_{g\in T_n} d_T (g y).
 \end{equation}
 In fact, from this, without loss of generality, we may assume that $\emptyset\neq T_n\subseteq F_n$ satisfies that $\lim\limits_{n\rightarrow \infty} \frac{|T_n|}{|F_n|}= 1$ and
 \eqref{1008292331}
 holds for $\nu$-a.e. $y\in Y$. Now applying Lemma \ref{1103082044} one sees that $\{T_n: n\in \N\}$ is a tempered F\o lner sequence of $G$, and so
 \begin{equation*}
  \limsup_{n\rightarrow \infty} \frac{1}{|F_n|} d_{F_n} (y)\le \frac{1}{|T|} \E (d_T| \mathcal{I}) (y)\ \text{(using Theorem \ref{1006131615})}
 \end{equation*}
 for $\nu$-a.e. $y\in Y$. Thus,
 \begin{equation} \label{1105292145}
 \limsup_{n\rightarrow \infty} \frac{1}{|F_n|} d_{F_n} (y)\le \inf_{T\in \mathcal{T}_G} \frac{1}{|T|} \E (d_T| \mathcal{I}) (y)
 \end{equation}
 for $\nu$-a.e. $y\in Y$.
 We should observe first (using Proposition \ref{1103051237}) that:
 \begin{equation} \label{1103082137}
 \int_Y \inf_{T\in \mathcal{T}_G} \frac{1}{|T|} \E (d_T| \mathcal{I}) (y) d \nu (y)\le \inf_{T\in \mathcal{T}_G} \int_Y \frac{1}{|T|} \E (d_T| \mathcal{I}) (y) d \nu (y)= \nu (\mathbf{D}).
 \end{equation}
 Combined with \eqref{1103081842} and \eqref{1105292145}, we obtain a stronger version of \eqref{1103081657}:
 \begin{equation*}
 \int_Y \limsup_{n\rightarrow \infty} \frac{1}{|F_n|} d_{F_n} (y) d \nu (y)= \int_Y \inf_{T\in \mathcal{T}_G} \frac{1}{|T|} \E (d_T| \mathcal{I}) (y) d \nu (y)= \nu (\mathbf{D}).
 \end{equation*}

Clearly, the function $\inf\limits_{T\in \mathcal{T}_G} \frac{1}{|T|} \E (d_T| \mathcal{I}) (y)$ is measurable and $G$-invariant on $(Y, \mathcal{D}, \nu)$.

It remains to prove \eqref{1103081924}.
 Applying the ergodic decomposition, we may assume without loss of generality that the MDS $(Y, \mathcal{D}, \nu, G)$ is ergodic. From the discussion above, one deduces that
 $$ \limsup_{n\rightarrow \infty} \frac{1}{|F_n|} d_{F_n} (y)= \inf_{T\in \mathcal{T}_G} \frac{1}{|T|} \E (d_T| \mathcal{I}) (y)= \nu (\mathbf{D})$$
  for $\nu$-a.e. $y\in Y$, no matter $\nu (\mathbf{D})= - \infty$ or $> - \infty$. This finishes the proof.
 \end{proof}

The following results, \cite[Lemma 9.3]{DZRDS} and \cite[Lemma 2.2.16]{MO} are used in the proof of Theorem \ref{1102001916}.

\begin{lem} \label{1103081931}
  Let $T, E\in \mathcal{F}_G$. Then $\sum\limits_{t\in T} 1_{t E}= \sum\limits_{g\in E} 1_{T g}$.
  \end{lem}

\begin{lem} \label{11103081701}
 Let $\mathbf{D}= \{d_F: F\in \mathcal{F}_G\}\subseteq L^1 (Y, \mathcal{D}, \nu)$ be a strongly sub-additive family. Assume that
 $1_E= \sum\limits_{i= 1}^n a_i 1_{E_i}$, where $E, E_1, \cdots, E_n\in \mathcal{F}_G$ and $a_1, \cdots, a_n> 0, n\in \N$. Then
 $d_E (y)\le \sum\limits_{i= 1}^n a_i d_{E_i} (y)$ for $\nu$-a.e. $y\in Y$.
 \end{lem}

Now we prove Theorem \ref{1102001916}.

\begin{proof}[Proof of Theorem \ref{1102001916}]
 The proof is similar to that of Theorem \ref{1103061916}.

Firstly, we may assume that the family $- \mathbf{D}$ is non-negative and obtain
 \begin{equation} \label{11103081842}
 \int_Y \limsup_{n\rightarrow \infty} \frac{1}{|F_n|} d_{F_n} (y) d \nu (y)\ge \nu (\mathbf{D}).
 \end{equation}
 Now let $T\in \mathcal{F}_G$ be fixed. As $\{F_n: n\in \mathbb{N}\}$ is a F\o lner sequence of $G$, for each $n\in \mathbb{N}$ we set $E_n= F_n\cap \bigcap\limits_{g\in T} g^{- 1} F_n\subseteq F_n$, then $\lim\limits_{n\rightarrow \infty} \frac{|E_n|}{|F_n|}= 1$.
 Now for each $n\in \mathbb{N}$, by the construction of $E_n$, $t E_n\subseteq F_n$ for any $t\in T$, then obviously there exist $E_1', \cdots, E_m'\in \mathcal{F}_G, m\in \{0\}\cup \N$ and rational numbers $a_1, \cdots, a_m> 0$ such that
 \begin{equation*}
 1_{F_n}= \frac{1}{|T|} \sum_{t\in T} 1_{t E_n}+ \sum_{j= 1}^m a_j 1_{E_j'}.
 \end{equation*}
 Using Lemma \ref{1103081931}, one has
 $\sum\limits_{t\in T} 1_{t E_n}= \sum\limits_{g\in E_n} 1_{T g}$, and so
 \begin{equation} \label{1102111848}
 1_{F_n}= \frac{1}{|T|} \sum\limits_{g\in E_n} 1_{T g}+ \sum_{j= 1}^m a_j 1_{E_j'},
 \end{equation}
 which implies that, for $\nu$-a.e. $y\in Y$,
 \begin{eqnarray} \label{1008300040}
 d_{F_n} (y)&\le & \frac{1}{|T|} \sum_{g\in E_n} d_{T g} (y)+ \sum_{j= 1}^m a_j d_{E_j'} (y)\nonumber \\
 & & (\text{using Lemma \ref{11103081701}, as the family $\mathbf{D}$ is strongly sub-additive})\nonumber \\
 &\le & \frac{1}{|T|} \sum_{g\in E_n} d_{T g} (y)\ (\text{as the family $-\mathbf{D}$ is non-negative})\nonumber \\
 &= & \frac{1}{|T|} \sum_{g\in E_n} d_T (g y)\ (\text{as the family $\mathbf{D}$ is $G$-invariant}).
 \end{eqnarray}
 By Lemma \ref{1103082044}, $\{E_n: n\in \N\}$ is also a tempered F\o lner sequence of $G$, and hence by a similar argument to the proof of Theorem \ref{1103061916} we obtain the conclusion.
 \end{proof}

\section{The main technical results} \label{techn}

In this section, we aim to strengthen the results of the previous section
 for suitably well-behaved infinite countable discrete amenable
 groups.

First, we need to recall the well-known notion of a self-similar tiling.

Let $T\in \mathcal{T}_G$. We say that \emph{$T$ tiles $G$ self-similarly} if by a suitable selection, $G_T$ is a subgroup of $G$ isomorphic to $G$ via a group isomorphism $\pi_T: G\rightarrow G_T$.

Then we have the following useful observation.

\begin{prop} \label{1103061653}
 Let $\{F_n: n\in \N\}$ be a (tiling) F\o lner sequence of $G$ and suppose that $T_1, T_2\in \mathcal{T}_G$ tile $G$ self-similarly. If there exists $T\in \mathcal{T}_{G_{T_1}}$ such that $\{T g: g\in G_{T_2}\}$ forms a partition of $G_{T_1}$ then $\{T \pi_{T_2} (F_n): n\in \mathbb{N}\}$ is a (tiling) F\o lner sequence of $G_{T_1}$.
 \end{prop}
 \begin{proof}
 The tiling property is easy to check once $F_n, n\in \N$ has the tiling property.

Now we prove the asymptotic invariance property.
 Let $g\in G_{T_1}$. As $T\in \mathcal{T}_{G_{T_1}}$ such that $\{T g: g\in G_{T_2}\}$ forms a partition of $G_{T_1}$, then for each $t\in T$ there exist $g_t\in G_{T_2}$ and $t_g\in T$ such that $g t= t_g g_t$,
 moreover, if $t_1$ and $t_2$ are different elements from $T$ then $(t_1)_g\neq (t_2)_g$, otherwise $g_{t_1}\neq g_{t_2}$ and
 $$\emptyset\neq g T g_{t_1}^{- 1}\cap g T g_{t_2}^{- 1}= g (T g_{t_1}^{- 1}\cap T g_{t_2}^{- 1})\ \text{and so}\ T g_{t_1}^{- 1}\cap T g_{t_2}^{- 1}\neq \emptyset,$$
 a contradiction to the assumption that $\{T g: g\in G_{T_2}\}$ forms a partition of $G_{T_1}$ and the selection of $g_{t_1}, g_{t_2}\in G_{T_2}$. That is, there exists a bijection $f: T\rightarrow T$ such that $g t= f (t) g_t$ for some $g_t\in G_{T_2}$.

Thus for each $n\in \N$ one has
 \begin{eqnarray*}
 T \pi_{T_2} (F_n)\Delta g T \pi_{T_2} (F_n)&= & \bigcup_{t\in T} t \pi_{T_2} (F_n)\Delta \bigcup_{t\in T} t g_{f^{- 1} (t)} \pi_{T_2} (F_n) \\
 &\subseteq & \bigcup_{t\in T} t (\pi_{T_2} (F_n)\Delta g_{f^{- 1} (t)} \pi_{T_2} (F_n)).
 \end{eqnarray*}
 Observe that $\{F_n: n\in \N\}$ is a F\o lner sequence of $G$, and so $\{\pi_{T_2} (F_n): n\in \N\}$ is a F\o lner sequence of $G_{T_2}$. Thus
 $$\lim_{n\rightarrow \infty} \frac{|T \pi_{T_2} (F_n)\Delta g T \pi_{T_2} (F_n)|}{|T \pi_{T_2} (F_n)|}\le \lim_{n\rightarrow \infty} \frac{1}{|T|} \sum_{t\in T} \frac{|\pi_{T_2} (F_n)\Delta g_{f^{- 1} (t)} \pi_{T_2} (F_n)|}{|\pi_{T_2} (F_n)|}= 0.$$
 That is, $\{T \pi_{T_2} (F_n): n\in \mathbb{N}\}$ is also a F\o lner sequence of $G_{T_1}$.
 \end{proof}

As a direct corollary, we have:

\begin{cor} \label{1103051653}
 Let $\{F_n: n\in \N\}$ be a (tiling) F\o lner sequence of $G$ and $T\in \mathcal{T}_G$ tile $G$ self-similarly. Then $\{T \pi_T (F_n): n\in \mathbb{N}\}$ is a (tiling) F\o lner sequence of $G$.
 \end{cor}

Our main result, Theorem \ref{1102201550}, is stated below.  It generalizes Kingman's sub-additive ergodic theorem to suitably well-behaved infinite countable discrete amenable groups. Almost all known proofs of Kingman's sub-additive ergodic theorem rely heavily on algebraic structure of $\mathbb{Z}$ (or more precisely, the algebraic structure of the semigroup $\mathbb{N}$), see for example \cite{Krengel, Steele}, it seems natural to require some strong algebraic structure over the amenable groups considered in the theorem. Whilst the assumptions of the Theorem \ref{1102201550} may seem a little complicated, the best model to understand the assumptions is $\mathbb{Z}^d, d\in \mathbb{N}$ (see \S 6 for more such models).

  \begin{thm} \label{1102201550}
 Assume that $\mathbf{D}= \{d_F: F\in \mathcal{F}_G\}\subseteq L^1 (Y, \mathcal{D}, \nu)$ is a sub-additive $G$-invariant family satisfying $\nu (\mathbf{D})> - \infty$ and $\{F_n: n\in \mathbb{N}\}$ is a F\o lner sequence of $G$ consisting of self-similar tiling subsets.
  If, additionally,
   \begin{enumerate}

   \item[(a)] $\{F_n: n\in \mathbb{N}\}$ satisfies the Tempelman condition with constant $M> 0$ and

   \item[(b)] there exists an infinite subset $\mathcal{N}\subseteq \N$ such that,
   for each $m\in \mathcal{N}$, once $p\in \N$ is large enough then there exist $n_1, n_2\in \N$ such that $F_m \pi_{F_m} (F_{n_1})\supseteq F_p\supseteq F_m \pi_{F_m} (F_{n_2})$ and
 $\frac{|F_{n_1}|- |F_{n_2}|}{|F_p|}$
   is small enough.
   \end{enumerate}
   Then
   \begin{enumerate}

    \item \label{1103061740} There exists $d\in L^1 (Y, \mathcal{D}, \nu)$ such that
  $$\lim_{n\rightarrow \infty} \frac{1}{|F_n|} d_{F_n} (y)= d (y)$$
   for $\nu$-a.e. $y\in Y$ and $\int_Y d (y) d \nu (y)= \nu (\mathbf{D})$.

  \item \label{1105292312} If the limit function $d$ is $G$-invariant, then, for $\nu$-a.e. $y\in Y$,
            \begin{equation*} \label{1105292248}
            d (y)= \inf_{T\in \mathcal{T}_G} \frac{1}{|T|} \E (d_T| \mathcal{I}) (y).
            \end{equation*}

           \item \label{1103061741} If, in addition, one of the following conditions holds:
            \begin{enumerate}

           \item[(i)] The family $\mathbf{D}$ is $G$-bi-invariant.

           \item[(ii)] The family $\mathbf{D}$ is strongly sub-additive.

           \item[(iii)]
 There exist $p, m\in \mathcal{N}$ large enough (in the sense that once $N\in \N$ there exist such $p, m\in \mathcal{N}$ satisfying $p, m\ge N$) such that $G_{F_m} G_{F_p}= G$.
            \end{enumerate}
            Then the limit function $d$ is $G$-invariant.

       \item \label{1103062005} If, in addition, there exist sequences $\{r_1< r_2< \cdots\}\subseteq \mathcal{N}$ and $\{t_n: n\in \N\}$ such that $F_{r_{n+ 1}}= F_{r_n} \pi_{F_{r_n}} (F_{t_n})$ for each $n\in \N$, then, in the sense of $L^1$,
             $$\lim_{n\rightarrow \infty} \frac{1}{|F_n|} d_{F_n} (y)= d (y).$$
                         \end{enumerate}
    \end{thm}
 \begin{proof}
 We will follow the ideas of the proof of \cite[Chapter 1, Theorem 5.3]{Krengel}.

 As the F\o lner sequence $\{F_n: n\in \N\}$ satisfies the Tempelman condition, using Theorem \ref{1006131615} it is equivalent to consider the family $\mathbf{D}'= \{d'_F: F\in \mathcal{F}_G\}$ given by
 $$d_F' (y)= \sum_{g\in F} d_{\{e_G\}} (g y)- d_F (y)\ \text{for each}\ F\in \mathcal{F}_G.$$
 Then the family $\mathbf{D}'$ is non-negative, sup-additive and $G$-invariant and $\nu (\mathbf{D}')< \infty$.

For convenience, we may assume that the family $\mathbf{D}'$
 satisfies the assumptions of non-negativity, sup-additivity and
 $G$-invariance for each point $y\in Y$ without any ambiguity (for
 example, $d_{F g} (y)= d_F (g y)$ for each $F\in \mathcal{F}_G$ and
 any $y\in Y$).

Set $E= \{y\in Y: \overline{d}' (y)> \underline{d}' (y)\}$, where
 $$\overline{d}' (y)= \limsup_{n\rightarrow \infty} \frac{1}{|F_n|} d'_{F_n} (y)\ \text{and}\ \underline{d}' (y)= \liminf_{n\rightarrow \infty} \frac{1}{|F_n|} d'_{F_n} (y)\ge 0.$$
 Observe that by Fatou's Lemma one has
 \begin{equation} \label{1103072305}
 \int_Y \underline{d}' (y) d \nu (y)\le \liminf_{n\rightarrow \infty} \frac{1}{|F_n|} \int_Y d'_{F_n} (y) d \nu (y)= \nu (\mathbf{D}'),
 \end{equation}
 and so $\underline{d}' (y)\in L^1 (Y, \mathcal{D}, \nu)$, as $\nu (\mathbf{D})$ is finite, equivalently, $\nu (\mathbf{D}')$ is finite.

Now for each $m\in \mathcal{N}$ we introduce
 $$\mathbf{D}'_m= \{d'_{m, \pi_{F_m} (F)} (y)\doteq d'_{F_m \pi_{F_m} (F)} (y)- \sum_{g\in F} d'_{F_m} (\pi_{F_m} (g) y): F\in \mathcal{F}_G\}.$$
 Then the family $\mathbf{D}'_m$ is non-negative, sup-additive and $G_{F_m}$-invariant. Here, the $G_{F_m}$-invariance of $\mathbf{D}_m'$ means that, for all $F\in \mathcal{F}_G, g\in G_{F_m}$ and $y\in Y$,
 $$d'_{m, \pi_{F_m} (F) g} (y)= d'_{m, \pi_{F_m} (F)} (g y).$$

In the following, first we shall prove that
 \begin{equation} \label{1103051610}
 \limsup_{n\rightarrow \infty} \frac{1}{|F_m| |F_n|} d'_{F_m
 \pi_{F_m} (F_n)} (y)\ge \overline{d}' (y)
 \end{equation}
 and
 \begin{equation} \label{1103051612}
 \liminf_{n\rightarrow \infty} \frac{1}{|F_m| |F_n|} d'_{F_m
 \pi_{F_m} (F_n)} (y)\le \underline{d}' (y).
 \end{equation}

In fact, suppose that $\{k_1< k_2< \cdots\}, \{p_1< p_2<
 \cdots\}\subseteq \N$ such that
 $$\overline{d}' (y)= \lim_{n\rightarrow \infty} \frac{1}{|F_{k_n}|} d'_{F_{k_n}} (y)\ \text{and}\ \underline{d}' (y)= \lim_{n\rightarrow \infty} \frac{1}{|F_{p_n}|} d'_{F_{p_n}} (y).$$
 By assumption (b) for each $n\in \N$ large enough we can select $l_n, q_n\in \N$ such that
 \begin{equation} \label{1103051749}
 F_m \pi_{F_m} (F_{l_n})\supseteq F_{k_n}\ \text{and}\ \lim_{n\rightarrow \infty} \frac{|F_m| |F_{l_n}|}{|F_{k_n}|}= 1
 \end{equation}
 and
 \begin{equation} \label{1103051750}
 F_m \pi_{F_m} (F_{q_n})\subseteq F_{p_n}\ \text{and}\ \lim_{n\rightarrow \infty} \frac{|F_m| |F_{q_n}|}{|F_{p_n}|}= 1.
 \end{equation}
 As the family $\mathbf{D}'$ is non-negative and sup-additive (and hence monotone), one has:
 \begin{eqnarray*}
 \limsup_{n\rightarrow \infty} \frac{1}{|F_m| |F_n|} d'_{F_m \pi_{F_m} (F_n)} (y)&\ge & \limsup_{n\rightarrow \infty} \frac{1}{|F_m| |F_{l_n}|} d'_{F_m \pi_{F_m} (F_{l_n})} (y) \\
 &\ge & \limsup_{n\rightarrow \infty} \frac{1}{|F_{k_n}|} d'_{F_{k_n}} (y)\ \text{(using \eqref{1103051749})}
 \end{eqnarray*}
 and
 \begin{eqnarray*}
 \liminf_{n\rightarrow \infty} \frac{1}{|F_m| |F_n|} d'_{F_m \pi_{F_m} (F_n)} (y)&\le & \liminf_{n\rightarrow \infty} \frac{1}{|F_m| |F_{q_n}|} d'_{F_m \pi_{F_m} (F_{q_n})} (y) \\
 &\le & \liminf_{n\rightarrow \infty} \frac{1}{|F_{p_n}|} d'_{F_{p_n}} (y)\ \text{(using \eqref{1103051750})},
 \end{eqnarray*}
 which implies the inequalities \eqref{1103051610} and
 \eqref{1103051612}.

Observe again that the family $\mathbf{D}'$ is non-negative and
 sup-additive. Hence from \eqref{1103051610} and \eqref{1103051612}
 we have, for $\nu$-a.e. $y\in Y$,
 \begin{eqnarray}
 & &
 \overline{d}' (y)- \underline{d}' (y)\nonumber \\
 & &\hskip 16pt \le \limsup_{n\rightarrow \infty} \frac{1}{|F_m| |F_n|} d'_{F_m \pi_{F_m} (F_n)} (y)- \liminf_{n\rightarrow \infty} \frac{1}{|F_m| |F_n|} d'_{F_m \pi_{F_m} (F_n)} (y)\nonumber \\
 & &\hskip 16pt \le \limsup_{n\rightarrow \infty} \frac{1}{|F_m| |F_n|} d'_{F_m \pi_{F_m} (F_n)} (y)- \liminf_{n\rightarrow \infty} \frac{1}{|F_m| |F_n|}
 \sum_{g\in F_n} d'_{F_m} (\pi_{F_m} (g) y)\nonumber \\
 & &\hskip 16pt = \limsup_{n\rightarrow \infty} \frac{1}{|F_m| |F_n|} d'_{F_m \pi_{F_m} (F_n)} (y)- \lim_{n\rightarrow \infty} \frac{1}{|F_m| |F_n|}
 \sum_{g\in F_n} d'_{F_m} (\pi_{F_m} (g) y) \label{1103051854} \\
 & &\hskip 86pt \text{(by assumption (a), applying Theorem \ref{1006131615} to $d'_{F_m}$)}\nonumber \\
 & &\hskip 16pt = \limsup_{n\rightarrow \infty} \frac{1}{|F_m| |F_n|} [d'_{F_m \pi_{F_m} (F_n)} (y)- \sum_{g\in F_n} d'_{F_m} (\pi_{F_m} (g) y)]\nonumber \\
 & & \hskip 16pt \le \sup_{n\in \N} \frac{1}{|F_m| |F_n|} d'_{m, \pi_{F_m} (F_n)} (y). \label{1103051853}
 \end{eqnarray}

Moreover, by Theorem \ref{1006131615} the pointwise limit of the
 sequence
 $$\frac{1}{|F_m| |F_n|}
 \sum_{g\in F_n} d'_{F_m} (\pi_{F_m} (g) y)$$ exists (denote it by
 $d'_m$), is $G_{F_m}$-invariant and is dominated by $ \frac{1}{|F_m|
 |F_n|} d'_{F_m \pi_{F_m} (F_n)} (y)$. Combining \eqref{1103051610}
 and \eqref{1103051854} with \eqref{1103051853} we also obtain, for
 $\nu$-a.e. $y\in Y$,
 \begin{equation} \label{1103061754}
 0\le \overline{d}' (y)- d'_m (y)\le \sup_{n\in \N} \frac{1}{|F_m| |F_n|} d'_{m, \pi_{F_m} (F_n)} (y).
 \end{equation}

Applying Corollary \ref{1103072031} to $\mathbf{D}_m'$ we obtain,
 for each $\alpha> 0$,
 \begin{equation} \label{1103051815}
 \nu (\{y\in Y: \sup_{n\in \N} \frac{1}{|F_m| |F_n|} d'_{m, \pi_{F_m} (F_n)} (y)> \alpha\})\le \frac{M}{\alpha}\cdot \frac{\nu (\mathbf{D}'_m)}{|F_m|}.
 \end{equation}

\medskip

\noindent {\bf Step One: proof of \eqref{1103061740}}

\medskip

In the following,
 first we will prove $\nu (E)= 0$. Recall $E= \{y\in Y: \overline{d}' (y)> \underline{d}' (y)\}$.

Let $\epsilon> 0$. Obviously, once $m\in \mathcal{N}$ is
 sufficiently large then
 \begin{equation} \label{1103051524}
 \frac{1}{|F_m|} \int_Y d'_{F_m} (y) d \nu (y)> \nu (\mathbf{D}')- \epsilon.
 \end{equation}
 As $\{F_n: n\in \N\}$ is a F\o lner sequence of $G$ consisting of
 self-similar tiling subsets, by Corollary \ref{1103051653} the
 sequence $\{F_m \pi_{F_m} (F_n): n\in \N\}$ is a tiling F\o lner
 sequence of $G$, and so  by Proposition \ref{1103051237} one has
 \begin{eqnarray} \label{1103051520}
 \nu (\mathbf{D}'_m)&= & \lim_{n\rightarrow \infty} \frac{1}{|F_n|} \int_Y [d'_{F_m \pi_{F_m} (F_n)} (y)- \sum_{g\in F_n} d'_{F_m} (\pi_{F_m} (g) y)] d \nu (y)\nonumber \\
 &= & \lim_{n\rightarrow \infty} \frac{1}{|F_n|} \int_Y d'_{F_m \pi_{F_m} (F_n)} (y) d \nu (y)- \int_Y d'_{F_m} (y) d \nu (y)\nonumber \\
 &= & |F_m| \nu (\mathbf{D}')- \int_Y d'_{F_m} (y) d \nu (y)\le |F_m| \epsilon\ \text{(using \eqref{1103051524})}.
 \end{eqnarray}
 In particular, combining this with \eqref{1103051853} and
 \eqref{1103051815} we obtain
 \begin{eqnarray} \label{1103061854}
 & & \nu (\{y\in Y: \overline{d}' (y)- \underline{d}' (y)> \alpha\})\nonumber \\
 & & \hskip 26pt \le \nu (\{y\in Y: \sup_{n\in \N} \frac{1}{|F_m| |F_n|} d'_{m, \pi_{F_m} (F_n)} (y)> \alpha\})\le \frac{M \epsilon}{\alpha}.
 \end{eqnarray}
 First letting $\epsilon\rightarrow 0$ and then letting $\alpha\rightarrow 0$ we obtain $\nu (E)= 0$ and so $\overline{d}' (y)= \underline{d}' (y)$ (denoted by $d' (y)$) for $\nu$-a.e. $y\in Y$.

For each $m\in \mathcal{N}$ observe that by \eqref{1103061754} we have
 \begin{eqnarray} \label{1103061838}
 \int_Y \overline{d}' (y) d \nu (y)&\ge & \int_Y d'_m (y) d \nu (y)\nonumber \\
 &= & \frac{1}{|F_m|} \int_Y d'_{F_m} (y) d \nu (y)\ \text{(using Theorem \ref{1006131615})}.
 \end{eqnarray}
 Letting $m\rightarrow \infty$ we obtain $\int_Y d' (y) d \nu (y)\ge \nu (\mathbf{D}')$ and hence $\int_Y d' (y) d \nu (y)= \nu (\mathbf{D}')$ (using \eqref{1103072305}), equivalently, $\int_Y d (y) d \nu (y)= \nu (\mathbf{D})$, here, $d$ is the pointwise limit function of the sequence $\frac{1}{|F_n|} d_{F_n} (y)$.

\medskip

\noindent {\bf Step Two: proof of \eqref{1105292312} }

\medskip

Applying the ergodic decomposition, without loss of generality, we may assume that the MDS $(Y, \mathcal{D}, \nu, G)$ is ergodic. If the limit function $d$ is $G$-invariant, by \eqref{1103061740} one has that $d (y)= \nu (\mathbf{D})$ for $\nu$-a.e. $y\in Y$. For each $T\in \mathcal{T}_G$, observe that the measurable function $\E (d_T| \mathcal{I})$ is invariant over $(Y, \mathcal{D}, \nu)$, and so $\E (d_T| \mathcal{I}) (y)= \int_Y d_T d \nu$ for $\nu$-a.e. $y\in Y$. Thus, for $\nu$-a.e. $y\in Y$,
 $$\inf_{T\in \mathcal{T}_G} \frac{1}{|T|} \E (d_T| \mathcal{I}) (y)= \inf_{T\in \mathcal{T}_G} \frac{1}{|T|} \int_Y d_T d \nu= \nu (\mathbf{D})= d (y)\ (\text{using Proposition \ref{1103051237}}).$$

\medskip

\noindent {\bf Step Three: proof of
 \eqref{1103061741}}

\medskip

 Now we prove the $G$-invariance of the limit function $d'$ under the assumptions.

 If the family $\mathbf{D}$ is either $G$-bi-invariant or strongly sub-additive, this follows directly from Theorem \ref{1103061916} and Theorem \ref{1102001916}, respectively. Now we assume that (iii) holds.

Let $g\in G$. We aim to prove $\nu (E_g)= 0$, where $E_g= \{y\in Y: d' (y)\neq d' (g y)\}$.

Let $\epsilon> 0$. By (iii), there exist $p, m\in \mathcal{N}$ sufficiently large such that $G_{F_m} G_{F_p}= G$. Thus
 \begin{equation} \label{1103061524}
 \frac{1}{|F_m|} \int_Y d'_{F_m} (y) d \nu (y)> \nu (\mathbf{D}')- \epsilon,
 \end{equation}
 \begin{equation} \label{1103061534}
 \frac{1}{|F_p|} \int_Y d'_{F_p} (y) d \nu (y)> \nu (\mathbf{D}')- \epsilon
 \end{equation}
 and there exist $g_m\in G_{F_m}$ and $g_p\in G_{F_p}$ such that $g= g_m g_p$.
 Observe that $d_m' (y)= d'_m (g_m y)$ and $d'_p (y)= d'_p (g_p y)$ for $\nu$-a.e. $y\in Y$, thus
 \begin{eqnarray} \label{1103051901}
 & &
 \nu (\{y\in Y: |d' (y)- d' (g y)|> 4 \alpha\})\nonumber \\
 &= & \nu (\{y\in Y: |d' (y)- d' (g_m g_p y)|> 4 \alpha\})\nonumber \\
 &\le & \nu (\{y\in Y: |d' (y)- d'_p (y)|> \alpha\})+ \nu (\{y\in Y: |d'_p (g_p y)- d' (g_p y)|> \alpha\})+\nonumber \\
 & & \hskip 66pt \nu (\{y\in Y: |d' (g_p y)- d'_m (g_p y)|> \alpha\})+\nonumber \\
  & & \hskip 66pt \nu (\{y\in Y: |d'_m (g_m g_p y)- d' (g_m g_p y)|> \alpha\})\nonumber \\
 &\le & \frac{4 M \epsilon}{\alpha}\ \text{(similar to reasoning of \eqref{1103061854}, using \eqref{1103061754}, \eqref{1103061524} and \eqref{1103061534})}\nonumber
 \end{eqnarray}
 for any $\alpha> 0$. First letting $\epsilon\rightarrow 0$ and then letting $\alpha\rightarrow 0$ we obtain $d' (y)= d' (g y)$ for $\nu$-a.e. $y\in Y$. In other words, $\nu (E_g)= 0$.

\medskip

\noindent {\bf Step Four: proof of
 \eqref{1103062005}}

\medskip

First, we claim that,
 by assumption, the sequence of the functions $\{d_{r_n}': n\in \N\}$ increases to some $d'_\infty\in L^1 (Y, \mathcal{D}, \nu)$.

 Let $n\in \N$ be fixed. As $F_{r_{n+ 1}}= F_{r_n} \pi_{F_{r_n}} (F_{t_n})$, in particular, $\{\pi_{F_{r_n}} (F_{t_n}) g: g\in G_{F_{r_{n+ 1}}}\}$ forms a partition of $G_{F_{r_n}}$ and $|F_{r_{n+ 1}}|= |F_{r_n}|\cdot |\pi_{F_{r_n}} (F_{t_n})|$ and so (recall that the family $\mathbf{D}'$ is sup-additive)
 \begin{eqnarray} \label{1103062107}
 & & \frac{1}{|F_{r_{n+ 1}}| |F_m|}
 \sum_{g\in F_m} d'_{F_{r_{n+ 1}}} (\pi_{F_{r_{n+ 1}}} (g) y)\nonumber \\
 & & \hskip 16pt \ge \frac{1}{|F_{r_n}|\cdot |\pi_{F_{r_n}} (F_{t_n})|\cdot |F_m|}
 \sum_{g\in \pi_{F_{r_n}} (F_{t_n}) \pi_{F_{r_{n+ 1}}} (F_m)} d'_{F_{r_n}} (g y)\nonumber \\
 & & \hskip 16pt =  \frac{1}{|F_{r_n}|\cdot |\pi_{F_{r_n}} (F_{t_n}) \pi_{F_{r_{n+ 1}}} (F_m)|}
 \sum_{g\in \pi_{F_{r_n}} (F_{t_n}) \pi_{F_{r_{n+ 1}}} (F_m)} d'_{F_{r_n}} (g y)
 \end{eqnarray}
  for each $m\in \N$. Observe that by assumptions, $\{\pi_{F_{r_n}} (F_{t_n}) \pi_{F_{r_{n+ 1}}} (F_m): m\in \N\}$ is a F\o lner sequence of $G_{F_{r_n}}$ (using Proposition \ref{1103061653}), and so by Theorem \ref{1006131615} one has that (to obtain \eqref{1103062116}, we may take a sub-sequence of $m\in \N$ if necessary, note that, as remarked by \cite[Proposition 1.4]{L} every F\o lner sequence of $G$ contains a tempered sub-sequence) both
  \begin{equation} \label{1103062115}
  \lim_{m\rightarrow \infty} \frac{1}{|F_{r_{n+ 1}}| |F_m|}
 \sum_{g\in F_m} d'_{F_{r_{n+ 1}}} (\pi_{F_{r_{n+ 1}}} (g) y)= d'_{r_{n+ 1}} (y)
  \end{equation}
 and
  \begin{equation} \label{1103062116}
 \lim_{m\rightarrow \infty} \frac{1}{|F_{r_n}|\cdot |\pi_{F_{r_n}} (F_{t_n}) \pi_{F_{r_{n+ 1}}} (F_m)|}
 \sum_{g\in \pi_{F_{r_n}} (F_{t_n}) \pi_{F_{r_{n+ 1}}} (F_m)} d'_{F_{r_n}} (g y)= d'_{r_n} (y)
  \end{equation}
  for $\nu$-a.e. $y\in Y$ and in the sense of $L^1$. Combining \eqref{1103062107} with
  \eqref{1103062115} and \eqref{1103062116} we obtain that the sequence of the functions $\{d_{r_n}': n\in \N\}$ increases.
  Let $d'_\infty$ be the limit function (which is non-negative). Observe that,
 by Theorem \ref{1006131615},
 $$\int_Y d'_m (y) d \nu (y)= \frac{1}{|F_m|} \int_Y d'_{F_m} (y) d \nu (y)\le \nu (\mathbf{D}')$$
 for each $m\in \mathcal{N}$. As $\nu (\mathbf{D})> -\infty$, equivalently, $\nu (\mathbf{D}')< \infty$, we obtain
 $$\int_Y d_\infty' (y) d \nu (y)\le \nu (\mathbf{D}')< \infty\ \text{and hence}\ \int_Y d_\infty' (y) d \nu (y)= \nu (\mathbf{D}'),$$
 in particular, $d'_\infty\in L^1 (Y, \mathcal{D}, \nu)$. Since
 $\int_Y d' (y) d \nu (y)= \nu (\mathbf{D}')$ and
 $d' (y)- d'_m (y)\ge 0$ for
 $\nu$-a.e. $y\in Y$ and each $m\in \mathcal{N}$, one has that $d' (y)=
 d_\infty' (y)$ for $\nu$-a.e. $y\in Y$.

To complete the proof, we only need to prove that
 $$\lim_{n\rightarrow \infty} \frac{1}{|F_n|} d'_{F_n} (y)= d'_\infty (y)$$
 in the sense of $L^1$.

Let $\epsilon> 0$. Obviously there exists $n\in \N$ such that
 \begin{equation} \label{1103062142}
 \int_Y |d'_\infty (y)- d'_{r_n} (y)| d \nu (y)< \epsilon
 \end{equation}
 and
 \begin{equation} \label{1103062143}
 |\int_Y d'_{r_n} (y) d \nu (y)- \nu (\mathbf{D}')|= |\frac{1}{|F_{r_n}|} \int_Y d'_{F_{r_n}} (y) d \nu (y) - \nu (\mathbf{D}')|< \epsilon.
 \end{equation}
  By our assumptions, once $m\in \N$ is large enough, there exists $s_m\in \N$ such that $F_m\supseteq F_{r_n} \pi_{F_{r_n}} (F_{s_m})$,
 \begin{equation} \label{1103062153}
 \frac{|F_m\setminus F_{r_n} \pi_{F_{r_n}} (F_{s_m})|}{|F_m|}< \epsilon
 \end{equation}
 and
 \begin{equation} \label{1103062156}
 \int_Y |\frac{1}{|F_{r_n}|\cdot |F_{s_m}|} \sum_{g\in F_{s_m}} d'_{F_{r_n}} (\pi_{F_{r_n}} (g) y)- d'_{r_n} (y)| d \nu (y)< \epsilon\ (\text{using \eqref{1103062115}}).
 \end{equation}
 Thus, once $m\in \N$ is large enough we have
 \begin{equation} \label{1103062152}
 0\le d_{F_m}' (y)- \sum_{g\in F_{s_m}} d'_{F_{r_n}} (\pi_{F_{r_n}} (g) y)
 \end{equation}
 and
 \begin{eqnarray} \label{1103062148}
 & &\hskip -26pt \int_Y [d_{F_m}' (y)- \sum_{g\in F_{s_m}} d'_{F_{r_n}} (\pi_{F_{r_n}} (g) y)] d \nu (y)\nonumber \\
 &\le & |F_m| \nu (\mathbf{D}')- |F_{s_m}| \int_Y d'_{F_{r_n}} (y) d \nu (y)\nonumber \\
 &\le & |F_m| \nu (\mathbf{D}')- |F_{s_m}|\cdot |F_{r_n}| (\nu (\mathbf{D}')- \epsilon)\ (\text{using \eqref{1103062143}})\nonumber \\
 &= & |F_m\setminus F_{r_n} \pi_{F_{r_n}} (F_{s_m})| \nu (\mathbf{D}')+ |F_{s_m}|\cdot |F_{r_n}| \epsilon,
 \end{eqnarray}
 and so
 \begin{eqnarray} \label{1103062157}
 & &\hskip -26pt \int_Y |\frac{1}{|F_m|} d_{F_m}' (y)- d'_\infty (y)| d \nu (y)\nonumber \\
 &\le & \frac{1}{|F_m|} \int_Y [d_{F_m}' (y)- \sum_{g\in F_{s_m}} d'_{F_{r_n}} (\pi_{F_{r_n}} (g) y)] d \nu (y)+\nonumber \\
 & & |\frac{1}{|F_m|}- \frac{1}{|F_{r_n}|\cdot |F_{s_m}|}| \int_Y \sum_{g\in F_{s_m}} d'_{F_{r_n}} (\pi_{F_{r_n}} (g) y) d \nu (y)\nonumber \\
 & & \int_Y |\frac{1}{|F_{r_n}|\cdot |F_{s_m}|} \sum_{g\in F_{s_m}} d'_{F_{r_n}} (\pi_{F_{r_n}} (g) y)- d'_{r_n} (y)| d \nu (y)+\nonumber \\
 & & \int_Y |d'_\infty (y)- d'_{r_n} (y)| d \nu (y)\nonumber \\
 &\le & \epsilon (2 \nu (\mathbf{D}')+ 3)\ \text{(using \eqref{1103062142}, \eqref{1103062153}, \eqref{1103062156} and \eqref{1103062148})}.
 \end{eqnarray}
 Letting $\epsilon\rightarrow 0$ we obtain the conclusion. This completes the proof.
 \end{proof}

In fact, even if $\nu (\mathbf{D})= - \infty$ we have a similar result.

\begin{thm} \label{1103051928}
 Assume that $\mathbf{D}= \{d_F: F\in \mathcal{F}_G\}\subseteq L^1 (Y, \mathcal{D}, \nu)$ is a sub-additive $G$-invariant family satisfying $\nu (\mathbf{D})= - \infty$ and $\{F_n: n\in \mathbb{N}\}$ is a F\o lner sequence of $G$ consisting of self-similar tiling subsets satisfying the assumptions of (a) and (b) in Theorem \ref{1102201550}. Then
  \begin{enumerate}

    \item There exists a measurable function $d$ over $(Y, \mathcal{D}, \nu)$ such that
  $$\lim_{n\rightarrow \infty} \frac{1}{|F_n|} d_{F_n} (y)= d (y)$$
   for $\nu$-a.e. $y\in Y$ and $\int_Y d (y) d \nu (y)= - \infty$.

\item If the limit function $d$ is $G$-invariant, then, for $\nu$-a.e. $y\in Y$,
            \begin{equation*}
            d (y)= \inf_{T\in \mathcal{T}_G} \frac{1}{|T|} \E (d_T| \mathcal{I}) (y).
            \end{equation*}

           \item \label{1103081257} If, additionally, one of the following conditions holds:
            \begin{enumerate}

           \item[(i)] The family $\mathbf{D}$ is $G$-bi-invariant.

           \item[(ii)] The family $\mathbf{D}$ is strongly sub-additive.

           \item[(iii)]
 There exist $p, m\in \mathcal{N}$ large enough such that $G_{F_m} G_{F_p}= G$.
            \end{enumerate}
            Then the limit function $d$ is $G$-invariant.
 \end{enumerate}
 \end{thm}
 \begin{proof}
 With the help of Theorem \ref{1102201550}, the proof of the conclusion is straightforward.

As in the proof of Theorem \ref{1102201550}, we may assume that the family $- \mathbf{D}$ is non-negative.
 For each $N\in \N$, we consider the family
 $$\mathbf{D}^{(N)}= \{d^{(N)}_E (y)\doteq \max \{- N |E|, d_E (y)\}: E\in \mathcal{F}_G\}\subseteq L^1 (Y, \mathcal{D}, \nu).$$
 It is easy to check that $\mathbf{D}^{(N)}$ is a sub-additive
 $G$-invariant family satisfying $\nu (\mathbf{D}^{(N)})\ge - N$.
 Thus we can apply Theorem \ref{1102201550} to the family
 $\mathbf{D}^{(N)}$ to see that there exists $d^{(N)}\in L^1 (Y,
 \mathcal{D}, \nu)$ such that
 $$\lim_{n\rightarrow \infty} \frac{1}{|F_n|} d_{F_n}^{(N)} (y)= d^{(N)} (y)$$
 for $\nu$-a.e. $y\in Y$ and $\int_Y d^{(N)} (y) d \nu (y)= \nu (\mathbf{D}^{(N)})$. Clearly, the sequence of functions $\{d^{(N)} (y): N\in \N\}$ decreases and set $d$ to be the limit function of it. So, $d$ is a measurable function over $(Y, \mathcal{D}, \nu)$. Moreover,
 \begin{eqnarray*}
 & & \hskip -68pt \int_Y d (y) d \nu (y)= \inf_{N\in \N} \int_Y d^{(N)} (y) d \nu (y)= \inf_{N\in \N} \inf_{n\in \N} \frac{1}{|F_n|} \int_Y d_{F_n}^{(N)} (y) d \nu (y) \\
 &= & \inf_{n\in \N} \inf_{N\in \N} \frac{1}{|F_n|} \int_Y d_{F_n}^{(N)} (y) d \nu (y) \\
 &= & \inf_{n\in \N} \frac{1}{|F_n|} \int_Y d_{F_n} (y) d \nu (y)= \nu (\mathbf{D}).
 \end{eqnarray*}

Once the measurable function $d$ is $G$-invariant, as in the proof of Theorem \ref{1102201550}, it is not hard to obtain that, for $\nu$-a.e. $y\in Y$,
            \begin{equation*}
            d (y)= \inf_{T\in \mathcal{T}_G} \frac{1}{|T|} \E (d_T| \mathcal{I}) (y).
            \end{equation*}

We aim that, for $\nu$-a.e. $y\in Y$,
 \begin{equation} \label{1103081431}
 \lim_{n\rightarrow \infty} \frac{1}{|F_n|} d_{F_n} (y)= d (y)\ (= \inf_{N\in \N} \lim_{n\rightarrow \infty} \frac{1}{|F_n|} d_{F_n}^{(N)} (y)).
 \end{equation}
 As for $\nu$-a.e. $y\in Y$ the limit $\lim\limits_{n\rightarrow \infty} \frac{1}{|F_n|} d_{F_n}^{(N)} (y)$ exists for each $N\in \N$. Fix such a point. Thus if $d (y)> - \infty$, say $d (y)> - M$ for some $M\in \N$ then, once $n\in \N$ is large enough, then for any $N\ge M$, $d_{F_n}^{(N)} (y)> - |F_n| M$ (in particular, $d_{F_n} (y)= d_{F_n}^{(N)} (y)$) and so the limit of the sequence $\frac{1}{|F_n|} d_{F_n} (y)$ exists and equals $\lim\limits_{n\rightarrow \infty} \frac{1}{|F_n|} d_{F_n}^{(N)} (y)$ (and hence equals $d (y)$). If $d (y)= - \infty$, observe that it is almost direct to check that $\limsup\limits_{n\rightarrow \infty} \frac{1}{|F_n|} d_{F_n} (y)\le d (y)$. Summing up, we obtain \eqref{1103081431}.

Finally, we are to prove the $G$-invariance of the limit function $d$ under the assumptions. If the family $\mathbf{D}$ is either $G$-bi-invariant or strongly sub-additive, it follows directly from Theorem \ref{1103061916} and Theorem \ref{1102001916}, respectively. Now we assume that (iii) holds. Then $d^{(N)}$ is $G$-invariant for each $N\in \N$, which implies immediately the $G$-invariance of $d$. This ends the proof.
 \end{proof}

\begin{rem} \label{1103101901}
 The only place in the proofs of Theorem \ref{1102201550} and Theorem \ref{1103051928} where we used the assumption that $\{F_n: n\in \N\}$ satisfies the Tempelman condition, is in applying Corollary \ref{1103072031}. Thus we can weaken the assumption to $\{F_n: n\in \N\}$ being tempered if Corollary \ref{1103072031} holds for a tempered F\o lner sequence of some particular group. Moreover, we can  drop the assumption that $\{F_n: n\in \N\}$ is tempered if Theorem \ref{1006131615} and Corollary \ref{1103072031} both hold for any F\o lner sequence of a given group.
 \end{rem}

 Note that $\mathbf{D}^{(N)}$ (in the proof of Theorem \ref{1103051928}) need not to be strongly sub-additive even if $\mathbf{D}$ is strongly sub-additive.

Except for assumptions (i) and (ii), all other assumptions in Theorems depend only on the group $G$ and the F\o lner sequence $\{F_n: n\in \N\}$ of $G$ (independent of the family $\mathbf{D}$).

Note that if $G$ is abelian, (i) holds for every $G$-invariant family.

We end this section with some remarks.

\begin{rem} \label{1103091726}
 Let $N\in \N$. Assume that, for each $m= 1, \cdots, N$, $G_m$ is an infinite countable discrete amenable group with $\{F_n^{(m)}: n\in \N\}$ a F\o lner sequence satisfying the assumptions appearing in Theorems (with $M_m$ as its Tempelman condition constant). Then it is not hard to check that, $\bigotimes\limits_{m= 1}^N G_m$ (an infinite countable discrete amenable group) and
 $\{F_{(n_1, \cdots, n_N)}: (n_1, \cdots, n_N)\in \bigotimes\limits_{1}^N \N\}$ (naturally generating many F\o lner sequences of $\bigotimes\limits_{m= 1}^N G_m$) also satisfy these assumptions, where, the subset $F_{(n_1, \cdots, n_N)}$ is given by
  $\bigotimes\limits_{m= 1}^N F_{n_m}^{(m)}$
  for each $(n_1, \cdots, n_N)\in \bigotimes\limits_{1}^N \N$.

Observe that a sequence in $\bigotimes\limits_{1}^N \N$, $(n_1, \cdots, n_N)$ tends to $\infty$ if and only if
 each $n_i$ increases to $\infty$: if some $n_i$, say $n_1$, is bounded, the problem is reduced to the case $N- 1$. Moreover, $\prod\limits_{m= 1}^N M_m$ will be the Tempelman condition constant of such a F\o lner sequence from $\{F_{(n_1, \cdots, n_N)}: (n_1,\cdots, n_N)\in \bigotimes\limits_{1}^N \N\}$.
 \end{rem}

\begin{rem} \label{1103091506}
 For the case of $N= \infty$ in Remark \ref{1103091726}, if in addition,
 \begin{equation} \label{1103091744}
 M\doteq \prod_{m\in \N} M_m< \infty,
 \end{equation}
 then we can carry out a similar discussion with respect to $\bigoplus\limits_{m\in \N} G_m$ (also an infinite countable discrete amenable group) and $\{F_\mathbf{n}: \mathbf{n}\in \bigoplus\limits_{\N} \N\}$  (naturally generating many F\o lner sequences for $\bigoplus\limits_{m\in \N} G_m$), where
 $$\bigoplus\limits_{m\in \N} G_m= \bigcup_{N'\in \N} \{(g, e_{G_{N'+ 1}}, e_{G_{N'+ 2}}, \cdots): g\in \bigotimes\limits_{m= 1}^{N'} G_m\},$$
 $$\bigoplus\limits_{\N} \N= \{(n_1, \cdots, n_m)\in \N^m: m\in \N\},$$
 and then for each $\mathbf{n}\in \bigoplus\limits_{\N} \N$, say $\mathbf{n}= (n_1, \cdots, n_m)$ for some $m\in \N$,
 $$F_\mathbf{n}= \{(g, e_{G_{m+ 1}}, e_{G_{m+ 2}}, \cdots): g\in \bigotimes\limits_{i= 1}^m F_{n_i}^{(i)}\}.$$

Observe that a sequence in $\bigoplus\limits_{\N} \N$, $(n_1, \cdots, n_m)$ tends to $\infty$ if and only if both $m$ and each $n_i$ tend montonically to $\infty$:  $m$ is bounded, we reduce to the case of Remark \ref{1103091726}). Moreover, $M$ will be the Tempelman  constant for this F\o lner sequence.
 \end{rem}

\section{Direct applications of the main technical results} \label{appli}

In this section, we give some direct applications of the
 results of the previous sections. Although we only consider
 some special groups in this section, we believe that these results
 (and the ideas in proving them) will have wider applications.

\subsection{The case of $G= \Z^m, m\in \N$}\

As shown by Remark \ref{1103091726}, the case of $\Z^m,m\in \N$ is reduced to the case of $\Z$. If $G= \Z$ which is abelian,
 consider $F_n= \{0, \cdots, n- 1\}$ with $G_{F_n}= n \Z$ for each $n\in \N$, then $2$ is its associated Tempelman condition constant and
 $F_m \pi_{F_m} (F_n)= \{0, 1, \cdots, m n- 1\}= F_{m n}$ for
  $m, n\in \N$.

 Thus, applying Theorem \ref{1102201550} and Theorem \ref{1103051928} we obtain:

 \begin{thm} \label{1103102128}
  Let $G= \Z^m, m\in \N$, $(Y, \mathcal{D}, \nu, G)$ be an MDS and $\{F_\mathbf{n}: \mathbf{n}\in \bigotimes\limits_1^m \N\}$ the sequence introduced as in Remark \ref{1103091726}. Assume that $\mathbf{D}= \{d_F: F\in \mathcal{F}_G\}\subseteq L^1 (Y, \mathcal{D}, \nu)$ is a sub-additive $G$-invariant family. Then, for $\nu$-a.e. $y\in Y$,
  \begin{equation} \label{1105242104}
  \lim\limits_{\mathbf{n}\in \bigotimes\limits_1^m \N\rightarrow \infty} \frac{1}{|F_\mathbf{n}|} d_{F_\mathbf{n}} (y)= \inf_{T\in \mathcal{T}_G} \frac{1}{|T|} \E (d_T| \mathcal{I}) (y),
  \end{equation}
 which is an invariant measurable function over $(Y, \mathbf{D}, \nu)$, and
 $$\int_Y \inf_{T\in \mathcal{T}_G} \frac{1}{|T|} \E (d_T| \mathcal{I}) (y) d \nu (y)= \nu (\mathbf{D}).$$ Moreover, if $\nu (\mathbf{D})> - \infty$ then \eqref{1105242104} also holds in the sense of $L^1$.
  \end{thm}

\subsection{The case of $G= \bigoplus\limits_{\N} K_n$ with each $K_n$ a non-trivial finite group}\

First, we consider the case where each $K_n, n\in \N$ is equal to a
 fixed non-trivial finite abelian group $K$ (a special case is
 $\bigoplus\limits_{\N} \Z_p, p\in \N\setminus \{1\}$, where $\Z_p$
 is the additive group $\{0, 1, \cdots, p- 1\}$). Obviously, $G$ is
 abelian. We consider $F_n= \{(g, e_K, e_K, \cdots): g\in
 \bigotimes\limits_1^n K\}$ with $G_{F_n}= \{(g_1, \cdots, g_n, g):
 g_1= \cdots= g_n= e_K, g\in \bigotimes\limits_\N K\}$ for each $n\in
 \N$. Then 1 is its associated Tempelman condition constant and $F_m
 \pi_{F_m} (F_n)= F_{m+ n}$ for
  $m, n\in \N$.
 A result similar to Theorem \ref{1103102128} holds.

For the case of $G= \bigoplus\limits_{n\in \N} K_n$ where each $K_n, n\in \N$ is a non-trivial finite abelian group: here $K_n$ need not be a fixed group. At first sight, it appears that we cannot apply our technical results directly. However,  $G$ is still an abelian group, and if we consider $F_n= \{(g, e_{K_{n+ 1}}, e_{K_{n+ 2}}, \cdots): g\in \bigotimes\limits_{i= 1}^n K_i\}$ with $G_{F_n}= \{(e_{K_1}, \cdots, e_{K_n}, g): g\in \bigotimes\limits_{i> n} K_i\}$ for each $n\in \N$, then 1 is also its associated Tempelman condition constant. Though in general $G_{F_n}$ is not isomorphic to $G$ via a group isomorphism, a rewriting of the argument of the previous section leads to the same conclusion as for the case of $\bigoplus\limits_{\N} K$ with $K$ a non-trivial finite abelian group.

Moreover, if we drop the assumption that each $K_n, n\in \N$ is abelian, we can still obtain the result if the family is strongly sub-additive.

\subsection{The case of $G= \bigoplus\limits_{\N} \Z$}\

For the abelian group $\bigoplus\limits_{\N} \Z$,
 we consider the sequence $\{F_\mathbf{n}: \mathbf{n}\in \bigoplus\limits_{\N} \N\}$, where, for $\mathbf{n}\in \bigoplus\limits_{\N} \N$, say $\mathbf{n}= (n_1, \cdots, n_m)$ for some $m\in \N$, the subset $F_\mathbf{n}$ is given as
 $$\{(g_1, \cdots, g_m, 0, 0, \cdots): g_i\in \{0, 1, \cdots, n_i- 1\}, i= 1, \cdots, m\}.$$
 As in Remark \ref{1103091506}, for $\bigoplus\limits_{\N} \N$, $(n_1, \cdots, n_m)$ tends to $\infty$ just means that both $m$ and all $n_i$ increasingly tend to $\infty$.
 For each $\mathbf{n}\in \bigoplus\limits_{\N} \N$, say $\mathbf{n}= (n_1, \cdots, n_m)$ for some $m\in \N$, obviously $F_\mathbf{n}$ tiles $\bigoplus\limits_{\N} \Z$ self-similarly with
 $$G_{F_\mathbf{n}}= \bigoplus_{i= 1}^m n_i \Z\bigoplus \bigoplus_\N \Z,$$
 and so for $\mathbf{n}'\in \bigoplus\limits_{\N} \N$, say $\mathbf{n}'= (n_1', \cdots, n_{m'}')$ for some $m'\in \N$, if $m'\ge m$ then
 $$F_\mathbf{n} \pi_{F_\mathbf{n}} (F_{\mathbf{n}'})= F_{\mathbf{n}^*}, F_\mathbf{n} F_{\mathbf{n}'}= F_{\mathbf{n}^{**}}, F_\mathbf{n}\cup F_{\mathbf{n}'}= F_{\mathbf{n}^{***}}$$
 with
 $$\mathbf{n}^*= (n_1 n_1', \cdots, n_m n_m', n_{m+ 1}', \cdots, n_{m'}'),$$
 $$\mathbf{n}^{**}= (n_1+ n_1'- 1, \cdots, n_m+ n_m'- 1, n_{m+ 1}', \cdots, n_{m'}'),$$
  $$\mathbf{n}^{***}= (\max \{n_1, n_1'\}, \cdots, \max \{n_m, n_m'\}, n_{m+ 1}', \cdots, n_{m'}').$$
  We have similar formulas in the case of $m'< m$.

It is easy to see that even in general a F\o lner sequence from $\{F_\mathbf{n}: \mathbf{n}\in \bigoplus\limits_{\N} \N\}$ is not tempered,
 and so we cannot apply directly any of the results of the previous sections to $\{F_\mathbf{n}: \mathbf{n}\in \bigoplus\limits_{\N} \N\}$.

Before proceeding, let us first recall some well-known results.

The first one can be found in any standard book about ergodic theory.

\begin{thm} \label{1103102009}
 Let $\vartheta$ be an invertible measure-preserving transformation over a Lebesgue space $(Y, \mathcal{D}, \nu)$ and $f\in L^p (Y, \mathcal{D}, \nu)$, where $1\le p< \infty$. Then
 $$\lim_{n\rightarrow \infty} \frac{1}{n} \sum_{i= 0}^{n- 1} f (\vartheta^i y)= \E (f| \mathcal{I}) (y)$$
 for $\nu$-a.e. $y\in Y$ and in the sense of $L^p$.
 \end{thm}

The second one is also standard: it is a variation of \cite[Chapter
 1, Theorem 5.2]{Krengel} (in fact, it is another version of the
 maximal ergodic inequality).

\begin{prop} \label{01103102009}
 Let $\vartheta$ be an invertible measure-preserving transformation over a Lebesgue space $(Y, \mathcal{D}, \nu)$ and $\mathbf{D}= \{d_F: F\in \mathcal{F}_\Z\}\subseteq L^1 (Y, \mathcal{D}, \nu)$ a non-negative sup-additive $\vartheta$-invariant family. Then, for each $\alpha> 0$,
 $$\nu (\{y\in Y: \sup_{n\in \N} \frac{1}{n} d_{\{0, 1, \cdots, n- 1\}} (y)> \alpha\})\le \frac{1}{\alpha} \nu (\mathbf{D}).$$
 \end{prop}

Now using Theorem \ref{1103102009} and Proposition \ref{01103102009}, and applying the results of the previous sections (especially, Remark \ref{1103101901}) we obtain a result similar to Theorem \ref{1103102128} (as $\bigoplus\limits_{\N} \Z$ is an abelian group).

\section*{Acknowledgements}

 The authors thank the referee for many important comments that
 have resulted in substantial improvements to this paper, in particular for informing
 us the paper of Pogorzelski \cite{Pogor}.

This work was largely carried out in the School of Mathematics
 and Statistics, University of New South Wales (Australia). We gratefully
 acknowledge the hospitality of UNSW. We also acknowledge the support of the
 Australian Research Council.

The third author was also supported by FANEDD (No. 201018), NSFC (No.
 10801035 and No. 11271078) and a grant from Chinese Ministry of Education (No.
 200802461004).

\vskip 16pt

\bibliographystyle{amsplain} 


\end{document}